\documentclass[10pt]{article}

\usepackage[top=1in,bottom=1in,left=1.25in,right=1.25in]{geometry}

\usepackage[absolute,overlay]{textpos}

\usepackage{graphicx}
\usepackage{ragged2e}
\usepackage{tabto}

\usepackage{amsmath}
\usepackage{amsthm}
\usepackage{amsfonts}
\usepackage{amssymb}

\usepackage{float}
\usepackage{multirow}
\usepackage{tikz}
\usepackage{graphics}
\usepackage{color}
\usepackage{url}
\usepackage{afterpage} 
\usepackage{multicol}
\usepackage[font=footnotesize]{caption}
\usepackage[font=footnotesize]{subcaption}
\usepackage{marginnote}
\usepackage{tcolorbox}
\usepackage[linesnumbered,boxed,algoruled,noline,noend]{algorithm2e}

\usepackage{lmodern}

\usepackage[sf,bf,medium]{titlesec}
\usepackage{pgfplots}

\pgfplotsset{compat=1.18}

\usepackage{bm}
\usepackage{booktabs}
\usepackage{siunitx}
\usepackage{isomath}

\usepackage[plainpages=false, colorlinks=true, citecolor=blue, filecolor=blue,
   linkcolor=blue, urlcolor=blue]{hyperref}
\usepackage{cleveref}

\usepackage{multicol}
\newcommand{\opnm}{\operatorname}

\newcommand{\lp}{\left(}
\newcommand{\rp}{\right)}

\usepackage[numbers]{natbib}

\newcommand\bbR{\mathbb R}
\newcommand\bbC{\mathbb C}

\newcommand\bx{\boldsymbol{x}}

\newcommand\by{\boldsymbol{y}}
\newcommand\bz{\boldsymbol{z}}

\newcommand\bA{\boldsymbol A}
\newcommand\bD{\boldsymbol D}

\newcommand\bK{\boldsymbol K}
\newcommand\bS{\boldsymbol S}

\newcommand\bp{\boldsymbol p}

\newcommand\bn{\boldsymbol n}

\newcommand\bu{\boldsymbol u}
\newcommand\bv{\boldsymbol v}

\newtheorem{theorem}{\sffamily Theorem}
\newtheorem{remark}{\sffamily Remark}

\newcommand{\cN}{\mathcal N}
\newcommand{\cW}{\mathcal W}

\newcommand{\cF}{\mathcal F}

\newcommand{\cI}{\mathcal I}
\newcommand{\cJ}{\mathcal J}

\newcommand{\cB}{\mathcal B}

\numberwithin{equation}{section}

\newcommand{\bs}{\boldsymbol}

\newcommand{\bsigma}{\bs\sigma}

\newcommand{\figref}[1]{\figurename~\hyperref[#1]{\ref{#1}}}
\newcommand{\tableref}[1]{\tablename~\hyperref[#1]{\ref{#1}}}

\renewcommand{\phi}{\varphi}

\binoppenalty=\maxdimen
\relpenalty=\maxdimen

\usepackage{setspace}
\usetikzlibrary{calc, decorations.markings, arrows.meta}
\usepackage{abstract}

\newcommand{\npts}{N_{\opnm{pts}}}

\newcommand{\nspread}{n_{\opnm{spread}}}
\newcommand{\nshell}{n_{\opnm{shell}}}
\newcommand{\nproxy}{n_{\opnm{proxy}}}
\newcommand{\npxy}{\nproxy}
\newcommand{\nnbr}{n_{\opnm{nbr}}}

\begin{document}

\newlength{\leggap}\setlength{\leggap}{0.1em}
\newlength{\capgap}\setlength{\capgap}{-1.1em}

{\centering {\Large 
A kernel-independent precorrected-FFT method
\\}
}

\begin{center}

 Tristan Goodwill\footnote{Department of Mathematics, University of Florida, Gainesville, FL
  (\texttt{tgoodwill@ufl.edu})} and Owen Melia\footnote{Center for Computational Mathematics, Flatiron Institute, New York, NY
  (\texttt{omelia@flatironinstitute.org})}\\
\end{center}

\begin{abstract}
The precorrected-FFT has been used to rapidly evaluate $N$-body sums arising in electromagnetics and acoustics settings. It has, however, been historically limited to kernels that satisfy a Green's identity. 
In this work, we demonstrate that, using an improved version of the method of proxy points, the precorrected-FFT can be used to efficiently compute $N$-body sums with a wide variety of kernels. This algorithm requires no analytic information beyond pointwise access to the kernel.
Our examples include kernels from non-local PDEs, Mat\'ern kernels, anisotropic kernels, and the Green's functions of high order PDEs.
We also provide a detailed algorithm for choosing the proxy points and spreading grid, as implemented in the new \texttt{pcfft} package.
We show these algorithmic choices result in favorable scaling when compared to other kernel-independent fast summation methods.
Finally, we use the algorithm to accelerate a wide variety of integral equations solvers involving non-standard kernels, with applications in surface scattering, geophysics, and lens design.
\end{abstract}
  

\section{Introduction}
\label{sec:intro}
In this work, we describe how the precorrected-FFT method~\cite{phillips_error_1996,phillips2002precorrected} can be used to compute $N$-body sums of the form
\begin{equation}\label{eq:Nbody}
    u(\bx_i) = \sum_{\by_j \neq \bx_i} K(\bx_i-\by_j) \boldsymbol{\sigma}_j
\end{equation}
for a very general class of kernels. Here, the kernel $K: \bbR^d \to \bbC$ is a smooth function away from the origin, and the source $\by_j$ and target $\bx_i$ points are general point clouds in $\bbR^d$, where the ambient dimension $d$ is $2$ or $3$.
$N$-body sums are ubiquitous in scientific computing, appearing in the numerical discretization of integral equation formulation of partial differential equations (PDEs) \cite{kress_linear_1999,nedelec_acoustic_2001,oneil_second-kind_2018-1,goodwillParametrixMethodElliptic2025}, Gaussian processes \cite{rasmussen_gaussian_2005,ambikasaranFastDirectMethods2016,mindenFastSpatialGaussian2017,wilsonKernelInterpolationScalable2015}, spatial statistics \cite{steinInterpolationSpatialData1999,chen_fast_2014}, molecular dynamics \cite{liangAcceleratingFastEwald2025,bostrom_fast_2026}, and many other settings. 
These applications often require repeated calculation of \cref{eq:Nbody}, for very large $N$, which denotes the number of source points $\{\by_j\}_{j=1,...,N}$ and target points $\{ \bx_i\}_{i=1,...,N}$. 
A na\"ive method of evaluating \cref{eq:Nbody} requires $O(N^2)$ work, while ``fast algorithms'' for this problem, our method included, approximate the sum using a method requiring $O(N)$ or $O(N\log(N))$ work. It is a nontrivial problem to approximate the $N$-body sum, because the kernel $K$ is non-smooth (or possibly diverging) when $\bx_i$ approaches $\by_j$, which prevents the use of global function approximation methods. 
At the same time, these kernels often decay slowly in $ \| \bx_i-\by_j \|$, meaning interaction distances are large and all $N^2$ terms in the sum must be addressed. The precorrected-FFT method \cite{phillips_error_1996,phillips2002precorrected} achieves this by approximating the potential from general source points with an equivalent set of charges on a regular grid, allowing the sum to be diagonalized with a pair of fast Fourier transform (FFT) and inverse transform (IFFT) operations. 
The nonuniform to regular approximation is only valid for sufficiently separated interactions, so interactions between nearby points must be ``pre-corrected''. 
While past work has focused on applying this method to individual PDE kernels, here we propose a ``kernel-independent'' verson of the algorithm which can automatically and efficiently apply to a broad class of problems given only pointwise access to the kernel. 
We also show how the method can be applied to a more general class of kernels using a recently-developed extension \cite{Askham_Goodwill_Hoskins_Nekrasov_Rachh_2025,goodwill_parametrix_2026} of the method of proxy points \cite{yeAnalyticalLowRankCompression2020}.
To emphasize the generality of the kernel-independent precorrected-FFT method, we provide numerical examples solving integral and integro-differential equations arising in biophysics, geophysics, and differential geometry.

The precorrected-FFT method is algorithmically similar to grid-based Ewald methods \cite{ewald_berechnung_1921,hockney_computer_1988,darden_particle_1993,liangAcceleratingFastEwald2025,bostrom_fast_2026}, which also spread nonuniform points to a regular grid, apply a convolution in Fourier space, and compute nearby interactions directly. For both precorrected-FFT and grid-based Ewald methods, the computationally intensive step is a FFT/IFFT pair, meaning both methods scale like $O(N\log(N))$. The conceptual difference between our proposed method and grid-based Ewald methods becomes apparent when one imagines the implementation. 
The prototypical five step Ewald method \cite{liangAcceleratingFastEwald2025} can only be implemented \emph{with a-priori knowledge of the kernel's decay in Fourier space}, which determines the number of required Fourier modes, and, indirectly, the number of regular gridpoints. 
The method we develop is complementary but serves a different purpose: a method which can automatically and efficiently determine the number of gridpoints and Fourier modes required for a user-specified error tolerance, using only pointwise access to the kernel. 
We envision this will be particularly useful for applications in integral equation formulations of PDEs, in settings where the Fourier transform of the integro-differential operator is complicated \cite{Askham_Goodwill_Hoskins_Nekrasov_Rachh_2025,herrera2026instability}, or the kernel changes during the runtime of the algorithm \cite{zhao_robust_2015,wangPeakHeatFlux2026}. 

The precorrected-FFT method was originally developed by Phillips and White \cite{phillips_error_1996,phillips2002precorrected} for the analysis of integrated circuits. It has gained popularity in circuit analysis \cite{hu_fast_2003}. Tissari and Rahola \cite{tissari_precorrected-fft_2003} use the method to accelerate forward solvers in a magnetoencephalography inverse problem in brain imaging. 
One of the authors of the present work has used the method in studies of surface wave interactions with floating plates and membranes \cite{Askham_Goodwill_Hoskins_Nekrasov_Rachh_2025}. 
The precorrected-FFT method has also been used to accelerate computations in nonlinear hydrodynamics \cite{yan2011efficient,jiang_precorrected-fft_2012}, acoustic wave propagation in shallow water \cite{li_development_2020}, acoustic scattering by impenetrable obstacles \cite{bruno2001fast,peng_precorrected-fft_2013,li_precorrected-fft_2022}, electromagnetic scattering by perfect conductors \cite{nie_precorrected-fft_2002,nie2002fast}, and electromagnetic analysis of thin multilayered conductors \cite{okhmatovski_three-dimensional_2009}. 
The precorrected-FFT method has been combined with a fast multipole method for large-scale aerodynamic simulations of aircraft \cite{willis_fastaero_2005}.

Fast algorithms for kernel summation rely on analytic properties of the kernel $K$ to form structured approximations of \cref{eq:Nbody} which can be exploited computationally. 
For instance, fast multipole methods \cite{greengard_fast_1987,carrier_fast_1988,ying_kernel-independent_2004} rely on a multipole or similar series expansion of $K$, which allows for rapid approximate interactions between sufficiently separated points. These methods incur complexity $O(N)$, and can adapt to highly-nonuniform distributions of points.
Jiang and Greengard's \cite{jiangDualspaceMultilevelKernelsplitting2025,klinteberg2026fast} recent dual-space multilevel kernel framework performs hierarchical kernel splits into smooth and local kernels, efficiently handling highly-nonuniform point distributions. 

While the fast multipole method enjoys an optimal computational complexity and adaptibility to challenging distributions of points, grid-based Ewald methods are often preferable because the constant factors implicit in the $O(N \log N)$ estimate are quite small, FFT software is highly-optimized for modern computer architectures, and many point distributions of interest are relatively benign enough that multiresolution schemes are not needed. 
In this paper, we show our kernel-independent precorrected-FFT method enjoys the same qualities, which lead to large speedups when used in integral equation solvers.

The remainder of this paper is organized in follows.
In \cref{sec:algo}, we describe the kernel-independent precorrected FFT algorithm in detail and give theoretical and empirical evidence for our choice of proxy points in various nonstandard kernels. A MATLAB implementation of the method as described in this section is available at \url{https://github.com/meliao/pcfft}. 
We summarize the complexity of the algorithm and run a series a scalability studies in \cref{sec:comp}.
In \cref{sec:numerics}, we provide a series of numerical examples where the precorrected-FFT is used to accelerate a wide variety of integral equations solvers. Finally, we then conclude with some closing remarks in \cref{sec:conc}.

\section{Algorithm description}\label{sec:algo}
Fast algorithms for computing $N$-body sums with generic point locations tend to come in one of two categories. The first, category, which includes the fast multipole method \cite{carrier_fast_1988,greengard_fast_1987} and Ewald splitting schemes \cite{ewald_berechnung_1921,hockney_computer_1988,darden_particle_1993,liangAcceleratingFastEwald2025,bostrom_fast_2026}, leverage an analytical understanding of the kernel, and typically explicit expressions, to achieve their speed. Leveraging this understanding gives these methods good computational speed, but comes at the cost of being kernel-specific. 
The other category includes fast algorithms that leverage the fact that for most kernels of interest the function $K(\bx-\by)$ can be well approximated by a low rank function for $(\by,\bx)$ in large swathes of the computational domain, typically for $\by$ in a ball and $\bx$ outside a larger ball. Sometimes this takes the form of polynomial series approximations (e.g. \cite{Fong_Darve_2009}), but a more popular method for translation-invariant kernels $K$ is to build an approximation of the form
\begin{align}\label{eq:approx_kern}
    K(\bx - \by) \approx \sum_{\bz_i \in \overline{\cB}(\by)} K(\bx - \bz_i ) \bs{\chi}_i (\by),
\end{align}
where $\overline{\cB}(\by)$ is a collection points near $\by$. 
This approximation can be thought of as ``spreading'' out the charge density $\boldsymbol{\sigma}_j\to \{\bs\chi_i(\bs y)\}$ and is the basis of a number of algorithms, including recursive skeletonization \cite{ho2012fast,martinsson2005fast,minden2017recursive}, kernel-independent FMMs \cite{malhotra2015pvfmm,ho_flam_2020}, and the precorrected-FFT method \cite{phillips_error_1996,phillips2002precorrected}. 
In the precorrected-FFT method, we lay down a regular grid of points $\bz_i$ that cover the computational domain and choose $\cN(\by)$ to be a subset of that grid living near $\by$. 
In \cref{sec:findpxy} we describe how we choose the subset and the equivalent weights $\bs{\chi}_i(\by)$ so that this approximation \cref{eq:approx_kern} holds for $\bx$ that is well-separated from $\by$. A major contribution of the present work is establishing this approximation for kernels $K$ which do not satisfy a Green's identity. 

If we neglect, for now, the fact that \cref{eq:approx_kern} only holds for some $\bx$, we arrive at the approximation
\begin{align*}
    \bu \approx \bK_{\bx \leftarrow \bz} \bS_{\by} \boldsymbol{\sigma},
\end{align*}
where $\left( \bK_{\bx \leftarrow \bz} \right)_{i,j} = K(\bx_i - \bz_j)$ is dense, 
and $\bS_{\by}$ is a sparse matrix of the spreading weights $\bs{\chi}_i(\by_j)$. An equivalent approximation to \cref{eq:approx_kern} will hold for $\bx$. Nesting the two approximations yields a sum of the form:
\begin{align*}
    \bu(\bx_k) &= \sum_{\by_l \neq \bx_k} K(\bx_k-\by_l) \boldsymbol{\sigma}_l \nonumber\\
    &\approx \sum_{\by_l\neq \bx_k}\lp\sum_{\bz_i \in \overline{\cB}(\by_l)} \left( \sum_{\bz_j \in \overline{\cB} (\bx_k)} K(\bz_j - \bz_i )  \tilde{\bs{\chi}}_j(\bx_k) \right) \bs{\chi}_i (\by_l)\rp \bs\sigma_l , \nonumber
\end{align*}
where again, we are temporarily ignoring the ``well-separated'' condition for \cref{eq:approx_kern}.
This nested summation can be organized into the composition of the ``spreading'' operation, a kernel summation on the approximation points $\bz_i$, and the adjoint of another ``spreading'' operation:
\begin{align}
    \bu \approx \bS_{\bx}^\top \bK_{\bz \leftarrow \bz} \bS_{\by} \bs{\sigma}
    \label{eq:far_approx}
\end{align}
The advantage of this expression is that the only remaining dense matrix, $\bK_{\bz \leftarrow \bz}$, has a multilevel Toeplitz structure because the $\bz_i$'s form a $d$-dimensional grid. This allows us to apply it quickly using a zero-padded fast Fourier transform \cite{frigoDesignImplementationFFTW32005}:
\begin{align*}
    \bK_{\bz \leftarrow \bz} \bv = \cF^{-1} \widehat{\bK_{\bz \leftarrow \bz}} \cF \bv,
\end{align*} 
where $\cF$ is a zero-padded $d$-dimensional fast Fourier transform, $\cF^{-1}$ is its inverse, and $\widehat{\bK_{\bz \leftarrow \bz}}$ is the Fourier transform of $K$ evaluated on the regular grid. \Cref{fig:matrix-decomp} shows the structure of the operators used to compute the approximate $N$-body sum in \cref{eq:far_approx}.

\DontPrintSemicolon
\SetArgSty{textnormal}
\begin{algorithm}[h!]
    \caption{
       \texttt{PCFFT\_APPLY}    }
    \label{alg:pcfft_apply}
    \KwIn{spreading matrices $\bS_{\bx},\bS_{\by}$;
    padded discrete Fourier transform of kernel $\widehat{\bK_{\bz \leftarrow \bz}}$;
    correction matrix $\bA$;
    source weights $\bs{\sigma}$.}
    Spread sources to the regular grid $\bS_{\by} \bs{\sigma}$ \\
    Evaluate discrete Fourier transform of sources $\texttt{IFFT}\left( \widehat{\bK_{\bz \leftarrow \bz}} \cdot \texttt{FFT}\left(  \bS_{\by} \bs{\sigma} \right)\right)$ \\
    Interpolate from regular grid to targets $ \bS_{\bx}^\top \texttt{IFFT}\left( \widehat{\bK_{\bz \leftarrow \bz}} \cdot \texttt{FFT}\left(  \bS_{\by} \bs{\sigma} \right)\right)$ \\
    Correct near interactions $\bA \bs{\sigma}$ \\
    $\bu \gets \bS_{\bx}^\top \texttt{IFFT}\left( \widehat{\bK_{\bz \leftarrow \bz}} \cdot \texttt{FFT}\left(  \bS_{\by} \bs{\sigma} \right)\right) + \bA \bs{\sigma}$ \\
    \KwResult{$\bu$}
\end{algorithm}

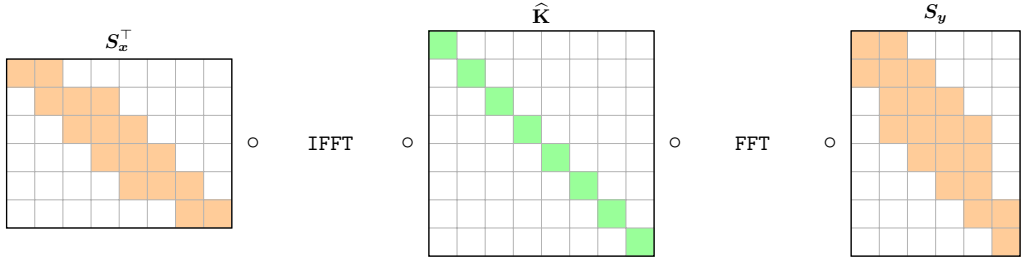
\begin{figure}[h!]
\centering
\resizebox{0.9\linewidth}{!}{
\begin{tikzpicture}[scale=0.45]
  \colorlet{sparsecol}{orange!40}
  \colorlet{fftcol}{blue!15}
  \colorlet{diagcol}{green!40}

  %

  \foreach \i/\jfrom/\jto in {0/0/1, 1/1/3, 2/2/4, 3/3/5, 4/4/6, 5/6/7} {
    \foreach \j in {\jfrom,...,\jto} {
      \fill[sparsecol] (\j, {-(\i+2)}) rectangle ++(1,1);
    }
  }
  \draw[gray!60,thin,step=1] (0,-7) grid (8,-1);
  \draw[black,semithick] (0,-7) rectangle (8,-1);
  \node[above,font=\small] at (4,-1) {$\bS_{\bx}^\top$};

  \node[font=\large] at (8.75,-4) {$\circ$};

  \node[font=\normalsize\sffamily] at (11.5,-4) {\texttt{IFFT}};

  \node[font=\large] at (14.25,-4) {$\circ$};

  \foreach \k in {0,1,2,3,4,5,6,7} {
    \fill[diagcol] ({15+\k},{-(\k+1)}) rectangle ++(1,1);
  }
  \draw[gray!60,thin,step=1] (15,-8) grid (23,0);
  \draw[black,semithick] (15,-8) rectangle (23,0);
  \node[above,font=\small] at (19,0) {$\mathbf{\widehat{K}}$};

  \node[font=\large] at (23.75,-4) {$\circ$};

  \node[font=\normalsize\sffamily] at (26.5,-4) {\texttt{FFT}};

  \node[font=\large] at (29.25,-4) {$\circ$};

  \foreach \i/\jfrom/\jto in {0/0/1, 1/0/2, 2/1/3, 3/1/4, 4/2/4, 5/3/4, 6/4/5, 7/5/5} {
    \foreach \j in {\jfrom,...,\jto} {
      \fill[sparsecol] ({30+\j},{-(\i+1)}) rectangle ++(1,1);
    }
  }
  \draw[gray!60,thin,step=1] (30,-8) grid (36,0);
  \draw[black,semithick] (30,-8) rectangle (36,0);
  \node[above,font=\small] at (33,0) {$\bS_{\by}$};

\end{tikzpicture}
}
\caption{The matrix decomposition in \cref{eq:far_approx}, read right to left: $\bS_{\by}$ (orange) is sparse with contiguous nonzeros per row;  $\widehat{K}$ (green) is diagonal, holding the Fourier transform of $K$ on the grid; and $\bS_{\bx}^\top$ (orange) is sparse with contiguous nonzeros per column.}
\label{fig:matrix-decomp}
\end{figure}

Of course, the approximations required to arrive at \cref{eq:far_approx} are only valid when $\bx$ and $\by$ are sufficiently well-separated. 
To that end, we introduce another sparse matrix, $\bA$, which corrects the interactions of nearby points where \cref{eq:approx_kern} was used erroneously. Thus, our approximation for general sets of source and target points is
\begin{align}
    \bu \approx \bS_{\bx}^\top \bK_{\bz \leftarrow \bz} \bS_{\by} \bs{\sigma} + \bA \bs{\sigma}.
    \label{eq:pcfft_approx}
\end{align}

\Cref{alg:pcfft_apply} describes how we evaluate the $N$-body sum once the matrices $\bS_{\bx}$ $\bS_{\by}$, $\widehat{\bK_{\bz \leftarrow \bz}}$, and $\bA$ are pre-computed. 
We describe the process of computing $\bA$ in \cref{sec:corrections}. 
We discuss our method for choosing the regular grid $\bz_i$ and spreading weights $\bs \chi_i(\bx)$ in \cref{sec:findpxy}. 
We test these spreading algorithms in  \cref{sec:testpxy} and introduce a method for applying the same approximation to derivatives of the kernel in \cref{sec:der}.

\subsection{Corrections for nearby targets}\label{sec:corrections}
The first term in \cref{eq:pcfft_approx} applies a far-field approximation to all of the pairs of points $(\bx_i, \by_j)$, but is only valid for those pairs of points which are ``well-separated''. 
To identify these points, we split the computational domain up into regular sized ``spreading bins" with side length $2h$. 
For any point $\by$ in a spreading bin $\cB_k \subset \bbR^d$, we will spread $\by$ to all the regular grid points that live in a ``spreading box'', $\overline{\cB}_{k}$, which is twice the size of the spreading bin; see \cref{fig:spreading_geometry}. 
We will say that a point $\bx$ is well-separated from $\cB_k$ if it is at least a distance $\rho_1=2\sqrt{d}h$ from the center of $\cB_k$. 
In other words, we will demand that \cref{eq:approx_kern} is valid for $\bx$ a distance $\rho_1$ from the center of $\cB_k$. \Cref{eq:far_approx} will therefore be valid when every point in $\overline{\cB}(\bx_k)$ is well separated from every point in $\overline{\cB}(\by_l)$. 
We can correct the errors made by \cref{eq:far_approx} by subtracting the contribution of this approximation for all pairs of points belonging to spreading bins whose centers lie within $2\rho_1$ of each other (\cref{fig:spreading_template}) and then adding the correct contributions. 
\begin{remark}
    A careful examination of \cref{fig:spreading_template} will show that we can actually use \cref{eq:far_approx} for boxes that are slightly closer than $2\rho_1$. We give ourselves the extra margin for stability, as it does not substantially increase the cost of the algorithm.
\end{remark}

In order to compute these contributions, we let 
\begin{align*}
    \cN_k = \left\{ {\bz\in \overline{\cB}_j} \ \  | \ \  \overline{\cB}_j \text{ is near } \cB_k \right\}
\end{align*}
be the neighborhood of the spreading bin $\cB_k$.  To evaluate the contribution of nearby points to \cref{eq:far_approx}, we construct a template for evaluating the kernel from a spreading box $\overline{\cB}$ to a neighborhood $\cN$:
\begin{align*}
     \left(\bK_{\cN \leftarrow \overline{\cB}}\right)_{i ,j } = K(\bz_j - \bz_i),
\end{align*}
where the points $\{ \bz_j \}$ are the regular discretization points in $\cN_k$ and the points $\{ \bz_i \}$ are the regular discretization points in $\overline{\cB}_k$.
Due to the translation invariance of the kernel $K(\cdot)$, we only need to precompute this matrix once, and we can use it for all bins $\cB_k$. 
We can now compute the correction matrix $\bA$:
\begin{align*}
    \bA(\cI_k, \cJ_k) =  \bD_k - \bS_{\bx}(:, \cI_k)^\top \bK_{\cN \leftarrow \overline{\cB}} \bS_{\by}(:, \cJ_k),
\end{align*}
where $\cI_k$ is the set of indices for target points $\bx_i$ in neighborhood $\cN_k$, $\cJ_k$ is the set of indices for source points $\by_j$ in spreading bin ${\cB}_k$, $\bD_k$ is the matrix of direct interactions between these points, and all other entries of $\bA$ are zero. 
\Cref{alg:compute_correction} summarizes this process of computing the correction matrix.

\newcommand{\halfside}{\texttt{halfside}}

\begin{figure}[h!]
\centering
\begin{subfigure}[b]{0.47\textwidth}
    \resizebox{\textwidth}{!}{
    \begin{tikzpicture}[scale=2]
  \draw[black!50, dashed] (-1,-1) rectangle (1,1);
  \draw[black!40, dashed] (-0.5,-0.5) rectangle (0.5,0.5);

  \draw[<->, thick] (-1,-1.25) -- (1,-1.25) node[midway, below] {Spreading box $\overline{\cB}$};

  \draw[<->, thick] (-0.5,-0.6) -- (0.5,-0.6) node[midway, below] {Spreading bin $\cB$};

  \foreach \x in {-0.9, -0.6, -0.3, 0, 0.3, 0.6, 0.9} {
    \foreach \y in {-0.9, -0.6, -0.3, 0, 0.3, 0.6, 0.9} {
      \fill[black] (\x,\y) circle (0.02);
    }
  }

  \draw[<->, thick] (-1.15,-0.85) -- (-1.15,0.85) node[midway, above, rotate=90] {$n_{\operatorname{spread}}$ points};

  \filldraw[red] (-0.3,-0.2) circle (0.04); 
  \filldraw[red] (-0.1,0.2) circle (0.04);
  \filldraw[red] (0.2,-0.1) circle (0.04);

  \draw[blue, thick] (0,0) circle (2);
  \foreach \t in {0,20,...,340} {
    \fill[blue] ( {2*cos(\t)}, {2*sin(\t)} ) circle (0.02);
  }

  \draw[<->, thick] (0,0) -- ({2*cos(45)},{2*sin(45)}) node[pos=0.3, above right, rotate=45] {Proxy radius $\rho_1$};

\end{tikzpicture}
    }
    \caption{The geometry of a spreading bin and box.}
    \label{fig:spreading_geometry}
\end{subfigure}
\begin{subfigure}[b]{0.47\textwidth}
    \resizebox{\textwidth}{!}{
    \begin{tikzpicture}
  \def\binSize{2cm}   
  \def\proxyRadius{1.05} 
  \def\ibin{3}
  \def\jbin{4}
  \def\boxMargin{0.2} 

  \begin{scope}[x=\binSize,y=\binSize]

  \coordinate (binCenter) at ({\ibin+0.5},{\jbin+0.5});

  \filldraw[fill=purple!50!white, draw=none] ({\ibin-1-\boxMargin},{\jbin+1-\boxMargin}) rectangle ++({1+2*\boxMargin},{1+2*\boxMargin});
  \filldraw[fill=purple!50!white, draw=none] ({\ibin-1-\boxMargin},{\jbin-\boxMargin})   rectangle ++({1+2*\boxMargin},{1+2*\boxMargin});
  \filldraw[fill=purple!50!white, draw=none] ({\ibin-1-\boxMargin},{\jbin-1-\boxMargin}) rectangle ++({1+2*\boxMargin},{1+2*\boxMargin});
  \filldraw[fill=purple!50!white, draw=none] ({\ibin-\boxMargin},  {\jbin+1-\boxMargin}) rectangle ++({1+2*\boxMargin},{1+2*\boxMargin});
  \filldraw[fill=purple!50!white, draw=none] ({\ibin-\boxMargin},  {\jbin-1-\boxMargin}) rectangle ++({1+2*\boxMargin},{1+2*\boxMargin});
  \filldraw[fill=purple!50!white, draw=none] ({\ibin+1-\boxMargin},{\jbin-1-\boxMargin}) rectangle ++({1+2*\boxMargin},{1+2*\boxMargin});
  \filldraw[fill=purple!50!white, draw=none] ({\ibin+1-\boxMargin},{\jbin-\boxMargin})   rectangle ++({1+2*\boxMargin},{1+2*\boxMargin});
  \filldraw[fill=purple!50!white, draw=none] ({\ibin+1-\boxMargin},{\jbin+1-\boxMargin}) rectangle ++({1+2*\boxMargin},{1+2*\boxMargin});
  \filldraw[fill=purple!50!white, draw=none] ({\ibin+2-\boxMargin},{\jbin-\boxMargin})   rectangle ++({1+2*\boxMargin},{1+2*\boxMargin});
  \filldraw[fill=purple!50!white, draw=none] ({\ibin-2-\boxMargin},{\jbin-\boxMargin})   rectangle ++({1+2*\boxMargin},{1+2*\boxMargin});
  \filldraw[fill=purple!50!white, draw=none] ({\ibin-\boxMargin},  {\jbin+2-\boxMargin}) rectangle ++({1+2*\boxMargin},{1+2*\boxMargin});
  \filldraw[fill=purple!50!white, draw=none] ({\ibin-\boxMargin},  {\jbin-2-\boxMargin}) rectangle ++({1+2*\boxMargin},{1+2*\boxMargin});

  \foreach \i in {0,...,6} {
    \foreach \j in {1,...,7} {
      \draw[black!50] (\i,\j) rectangle ++(1,1);
    }
  }

  \filldraw[fill=orange!20, draw=black] (\ibin,\jbin) rectangle ++(1,1);



  \draw[blue!70, dotted, thick] (binCenter) circle (\proxyRadius);

  \coordinate (bottomCenter) at ({\ibin+0.5},{\jbin-1.5});
  \coordinate (bottomRight) at ({\ibin+2.5},{\jbin-1.5});
  \draw[blue!70, dotted, thick] (bottomCenter) circle (\proxyRadius);
  \coordinate (bottomMid) at ({\ibin+1.5},{\jbin-1.5});
  \draw[blue!70, dotted, thick] (bottomMid) circle (\proxyRadius);


  \draw[<->, thin] (binCenter) -- ++(\proxyRadius,0) node[midway, above, font=\footnotesize] {Proxy radius};
  \end{scope}
\end{tikzpicture}
    }
    \caption{Illustrating the neighborhood of a spreading bin.}
    \label{fig:spreading_template}
\end{subfigure}
\caption{\Cref{fig:spreading_geometry}: the geometry of the spreading operation. Source or target points (red dots) belonging to the spreading bin are spread to the regular grid points (black dots) in the spreading box. 
The proxy surface is a constant radius $\rho_1$ away from the center of the spreading box. In this subfigure, $\nshell=1$, $\nproxy = 18$, and $\nspread=7$.
\Cref{fig:spreading_template}: our criteria for nearness is whether the inner proxy surfaces of two spreading bins intersect. The proxy surfaces of a few bins are shown in blue. The spatial support of bin $\cB_k$ is shown in orange, and the spatial support of its neighborhood $\cN_k$ is shown in pink.
}
\label{fig:spreading_diagram}
\end{figure}

\begin{algorithm}[h!]
    \caption{
       \texttt{COMPUTE\_CORRECTION}    }
    \label{alg:compute_correction}
    \KwIn{Target points $\{\bx_i \}_{i=1,\dots, N_{\text{src}} }$; source points $\{ \by_j\}_{j=1,\dots, N_{\text{targ}}}$; spreading matrices $\bS_{\by},\bS_{\bx}$; neighborhood template $\bK_{\cN \leftarrow \overline{\cB}}$; kernel $K(\cdot)$.
    }
    Initialize $\bA$ to zero.\\
    \For() {Bin $\cB_k \in \{ \cB_1, \cB_2, \dots, \cB_{N_{\text{bin}}} \}$}
    {
        Let $\cI_{k}$ index the target points in the neighborhood of $\cB_k$. \\
        Let $\cJ_{k}$ index the source points in spreading box $\overline{\cB}_k$. \\

    Compute exact interactions $\bD_{i,j} =  K(\bx_i, \by_j)$ for each $\bx_i \in \cI_{k}$ and $\by_j \in \cJ_{k}$. \\
    Compute approximate interactions $\bS_{\bx}(:, \cJ_k)^\top \bK_{\cN \leftarrow \overline{\cB}} \bS_{\by}(:, \cI_k)$.\\
    Add and subtract these matrices to the sparse matrix $\bA(\cI_k, \cJ_k) \gets \bD - \bS_{\bx}(:, \cJ_k)^\top \bK_{\cN \leftarrow \overline{\cB}} \bS_{\by}(:, \cI_k)$. \\
    }
    \KwResult{$\bA$}
\end{algorithm}

\subsection{Finding the spreading parameters}\label{sec:findpxy}
This section discusses how we automatically choose how many discretization points $\bz_j$ live in each $\cN(\by_k)$ and how to find the equivalent charges $\bs{\chi}_j$ so that \cref{eq:far_approx} holds for well-separated points. 
This method is designed so that for any distribution of source points $\by_i$ in a spreading bin $\cB_k$, we can automatically map charges $\bsigma$ to equivalent charges $\bs{\chi}$ supported on the pre-selected grid points $\bz_j \in \overline{\cB_k}$.

First, we choose $h$, the half side length of the spreading bin so that each points has around $\nnbr$ points in its bin. This parameter can be chosen to balance the cost of the FFTs against the cost of the sparse corrections $\bA$ (discussed below in \cref{sec:comp}). 
We find $h$ by constructing an adaptive hyperoctree of the union of source and target points such that no leaf box has more than $\nnbr/2^d$ points. We then set $h$ to be 
\begin{align*}
h = \begin{cases}\frac{\bar{h}}{4\sqrt{\pi}} \lp\frac{\npts}{\hat n}\rp^{1/2} & d = 2,\\
\frac{\bar{h}}{4\sqrt{\pi}} \left( \frac{3\sqrt{\pi}}{4} \right)^{1/3} \lp\frac{\npts}{\hat n}\rp^{1/3} & d=3,
\end{cases},
\end{align*}
where $\bar{h}$ is the mean leaf box side length, and $\hat n$ is the mean occupancy of each leaf box. 
In our implementation, we construct the adaptive hyperoctree using the utility provided in the \texttt{FLAM} package \citep{ho_flam_2020}.

Once we have determined $h$, and so the bins $\cB_k$, we turn to finding spreading weights $\bs\chi$. Given a user-specified error tolerance $\epsilon$, source points $\by \in \cB_k$, and target points $\bx$ that are well separated from the spreading bin, we want weights $\bs \chi$ which satisfy 
\begin{align}
    \frac{ \| \bK_{\bx \leftarrow \bz} \bs{\chi} - \bK_{\bx \leftarrow \by} \bs{\sigma} \|_2 }{\| \bK_{\bx \leftarrow \by} \bs{\sigma} \|_2} < \epsilon.
    \label{eq:spread_check}
\end{align}
To find these weights, we introduce a set of ``proxy points'', $\bp_j$, living outside of the spreading box $\overline{\cB_k}$. 
The proxy points give us a set of target points, at which the kernel evaluations must match between the original charges and the equivalent charges. This constraint can be written as the linear problem for $\bs{\chi}$:
\begin{align}\label{eq:proxy_constraint}
    \bK_{\bp \leftarrow \bz} \bs{\chi} &= \bK_{\bp \leftarrow \by} \bs{\sigma}.
\end{align}
We choose $\bs \chi$ as the minimum $\ell_2$ norm least squares solution of this equation by computing a complete orthogonal decomposition of $\bK_{\bp \leftarrow \bz}$, thresholding with tolerance $\epsilon$, and computing the least-squares solution.

The idea of using proxy points has a long history in the development of kernel-independent FMMs and broader classes of kernel matrix approximation methods; see \cite{yeAnalyticalLowRankCompression2020} for an overview. 
The use of proxy points is also standard in pre-corrected FFT algorithms. Phillips and White
\citep{phillips_error_1996,phillips2002precorrected} use proxy points supported on a sphere to define a set of linear constraints, which is solved in a least-squares sense via the pseudoinverse computed with the SVD. 
Error analysis is provided for Helmholtz and Laplace kernels.
\citet{nie2002fast} uses the same method applied to the electric field integral equation kernel, which includes a derivative. 
\citet{bruno2001fast} use proxy points on a sphere for the single- and double-layer Helmholtz kernels.
\citet{yeAnalyticalLowRankCompression2020} analyze the accuracy of the proxy point method for the Laplace kernel when the proxy points are chosen uniformly on the proxy surface; in this case they show the approximation error converges exponentially with the number of proxy points. 

In all of these works, the proxy points are arranged on a single curve or surface. A PDE stability theorem or Green's identity can be used to show that this is enough to guarantee accuracy of \cref{eq:approx_kern} outside of that shell. 
Indeed, if the proxy points form good interpolation nodes (we use equispaced points when $d=2$ and Fibonacci spiral points when $d=3$), \cref{eq:proxy_constraint} will guarantee that 
\begin{equation}
    K(\bx - \by) - \sum_{\bz_i\in \cN(\by)}K(\bx - \bz_i ) \bs{\chi}_i (\by),
\end{equation}
is a solution of the PDE with small Dirichlet data determined by the interpolation error. 
This will guarantee that it is small for all $\bx$ outside the shell.
For kernels based on higher order PDEs, PDE stability theorems require more than just Dirichlet boundary data. 
We discuss how we handle this case in \cref{sec:der}.

For kernels that do not arise from a Green's function of an elliptic PDE, 
we can not expect sufficient accuracy when using a single proxy surface. Instead, we enforce agreement on a set of $\nshell$ proxy surfaces, with radii $\rho_k$. We choose $\rho_1$ in the same way as described above and choose the remaining shells spaced so that $1/\rho_k$ are second-kind Chebyshev nodes. If $\{x_k\}_{k=2, \dots \nshell}$ are the second-kind Chebyshev points on the interval $(-1,1)$, then the radii are 
\begin{align*}
    \rho_k = \frac{2 \rho_1}{x_k + 1} = \frac{2\rho_1}{\cos\lp\frac{ (k-1)\pi}{\nshell-1}\rp + 1}, \quad k=2,\ldots, \nshell.
\end{align*}
This choice is motivated by \cite[Thrm. 4]{goodwill_parametrix_2026}, where the authors phrase the approximation \cref{eq:approx_kern} as interpolation problem of each angular Fourier mode of $K(\cdot -\by)$. They use this framing to show that if $K$ is some combination of derivatives of $\log(\cdot)$ in $\bbR^3$, then choosing the shells to be reciprocal Chebyshev nodes guarantees spectral convergence in the approximation \cref{eq:approx_kern} as $\nshell$ grows, up to the error determined by the number of nodes on each shell, $\nproxy$. Their proof technique can be directly extended to other kernels that have multipole expansions in a lower dimension and we find this choice works well for all of the kernels considered in this work. We add the extra shell at $\rho_1$ for stability and to provide uniformity with the single shell case. We find our choice is an improvement over the method suggested by \citet{Askham_Goodwill_Hoskins_Nekrasov_Rachh_2025}, who choose $\rho_k$ so that $1/\rho_k$ are equally-spaced.

Ultimately, our algorithm has three kernel-specific parameters:  the number of grid points per box in each direction, $\nspread$; the number of proxy points on each shell, $\nproxy$; and the number of proxy shells, $\nshell$. 
See \cref{fig:spreading_geometry} for an example of these parameters.
For individual kernels, analysis of multipole expansions allow us to derive how these parameters scale with tolerance $\epsilon$, using arguments similar to those in \cite{goodwill_parametrix_2026}. These estimates are listed in \cref{tab:typ_nspread} for a few representative examples.

\begin{table}[h!]
    \centering
    \begin{tabular}{lllll}
        \toprule
        Kernel & $d$ & $\nspread$ & $\nproxy$ & $\nshell$ \\
        \midrule
        $\log(r)$ & 2& $O\lp |\log \epsilon|^{1/2}\rp $ & $O\lp |\log \epsilon|\rp $&$O(1)$\\
        $\log(r)$ & 3& $O\lp |\log \epsilon|\rp $ & $O\lp |\log \epsilon|^{2}\rp $& $O\lp |\log \epsilon|\rp $\\
        $1/r$ & 3& $O\lp |\log \epsilon|^{2/3}\rp $ & $O\lp |\log \epsilon|^{2}\rp $&$O(1)$ \\
        \bottomrule
    \end{tabular}
    \caption{The dependence of $\nspread$, $\nproxy$, and $\nshell$ on $\epsilon$ for a few representative kernels.}
    \label{tab:typ_nspread}
\end{table}

While estimates exist for particular kernels, we use a purely empirical approach to choosing $(\nspread, \nproxy, \nshell)$ in this work, which we find to be robust and efficient. First, we choose $\nproxy$ so that $K(\bp - \by_j)$ is a well resolved function of $\bp$ on the shell of radius $\rho_1$ for a few sources $\by_j$ by checking the Fourier series or spherical harmonic tails in two or three dimensions, respectively. 
We do this using 117 random $\by_j$ inside a box with half side length $h$, as well as its corners. 
We then compute the coefficients of a series expansion of $K(\bp - \by)$ in a Fourier series ($d=2$) or a spherical harmonic series ($d=3$); we ensure there are enough discretization points $\bp$ to capture the tails of this series.
In both cases, we give ourselves a margin to account for the Lesbegue constant of the interpolation by multiplying the minimal $\nproxy$ by $(4/3)^{d-1}$, which corresponds to enforcing accuracy if we chop off the last quarter of the orders in the Fourier or spherical harmonic series expansion of $K(\bp-\by)$.

Once we have determined $\nproxy$, we sweep through possible pairs $(\nspread,\nshell)$ looking for an allowable pair.
For each proposed $\nspread$, we sweep up in $\nshell$ to see if \cref{eq:spread_check} is satisfied for 500 targets points with random radii in $(\rho_1,2\rho_1)$  and spread uniformly in angle. If none are found for $\nshell$ less than some maximum (typically 20), we increase $\nspread$ by one and then sweep through $\nshell$ again. If the kernel is known to satisfy a Green's identity, then we force $\nshell=1$ and simply sweep in $\nspread$. Our algorithm for computing $(\nspread,\nproxy,\nshell)$ is summarized in \cref{alg:nspread_nproxy}. 

\begin{remark}\label{rem:spikeshell}
    We will see below (e.g. in \cref{subfig:fig_log3d_1overr2d_nshell}) that this algorithm can occasionally lead to large $\nshell$ when a given $\nspread$ is barely enough to achieve a tolerances. We have found that this large $\nshell$ can be avoided by checking if there is a smaller number of shells that works with $\nspread+1$ grid points in each direction. In practice, we use  \cref{alg:nspread_nproxy} as it yields the smallest $\nspread$, which has a bigger impact on the code speed.
\end{remark}

\begin{algorithm}[h!]
    \caption{
       \texttt{NSPREAD\_NPROXY\_NSHELL}
       }
    \label{alg:nspread_nproxy}
    \KwIn{Spreading bin sidelength $h$; proxy radius $\rho$; requested tolerance $\epsilon$; kernel $K(\cdot)$; \texttt{multishell} flag.
    }
    If \texttt{multishell}, $n_{\operatorname{max}}\gets 20$, else $n_{\operatorname{max}}\gets 1$.\;
    $\by \gets$ 117 points randomly sampled uniform on $[-h,h]^d$ and append the corners. \;
    Find $\nproxy$ that resolves $K(\bp-\by)$ on the first proxy shell. \;
    $\nspread \gets 1$. \; 
    $\nshell \gets 1$. \;\label{ln:while} 
    While \cref{eq:spread_check} is unsatisfied, increase $\nshell$. If $\nshell>n_{\operatorname{max}}$, $\nspread \gets \nspread +1$ and return to  step \ref{ln:while}. If $\nspread>100$, abort.\; 
    \KwResult{$\nspread$; $\nproxy$; $\nshell$.}
\end{algorithm}

\subsection{Proxy point tests}\label{sec:testpxy}
In this section, we test \cref{alg:nspread_nproxy} and demonstrate how $\nspread,\npxy$, and $\nshell$ depend on $\epsilon$ for a variety of different kernels. In all of the experiments we choose the half side length $h=1$. In \cref{fig:laplace_helmholtz}, we show how $\nproxy$ and $\nspread$ depend on the tolerance for the Laplace and Helmholtz Green's functions in 2D and 3D. 
As predicted in \cref{tab:typ_nspread}, we see that $\nproxy$ scales as $|\log\epsilon|$ in 2D and $\log^2\epsilon$ in 3D, and that $\nspread$ grows a rate between $\sqrt{|\log\epsilon|}$ and $|\log\epsilon|$. We also see that $\nspread$ increases with the frequency of the Helmholtz problem. In \cref{fig:log3d_1overr2d} we repeat this experiment, but we swap the domain dimensions. This time the kernels do not satisfy Green's identities and so we run the algorithms in multi-shell mode. The figure shows that $\nproxy$ and $\nshell$ behave in a similar dimension-dependent way. We also see that the number of shells grows roughly as $|\log\epsilon|$ as the tolerance shrinks and that three dimensional kernels in 2D generally need a few more shells. The figures also show that the interaction rank with these kernels is slightly higher in this experiment than the previous one, as is exhibited by the increased $\nspread$ and number of proxy points. The large $\nshell$ at $\epsilon\sim 10^{-4}$ can be prevented using the approach described in \cref{rem:spikeshell}.

\begin{figure}
    \newlength{\panelw}\setlength{\panelw}{0.45\textwidth}

    \centering

    \begin{minipage}[b]{\panelw}\centering
        \includegraphics[width=\textwidth]{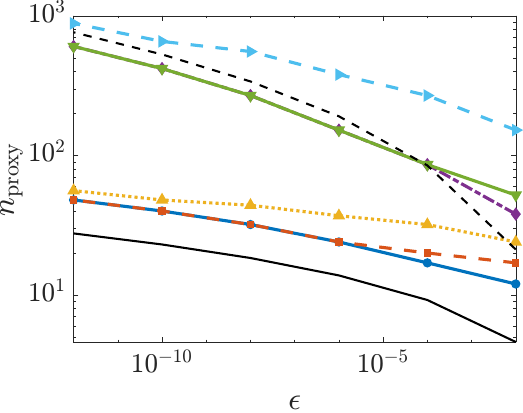}
    \end{minipage}\hfill
    \begin{minipage}[b]{\panelw}\centering
        \includegraphics[width=\textwidth]{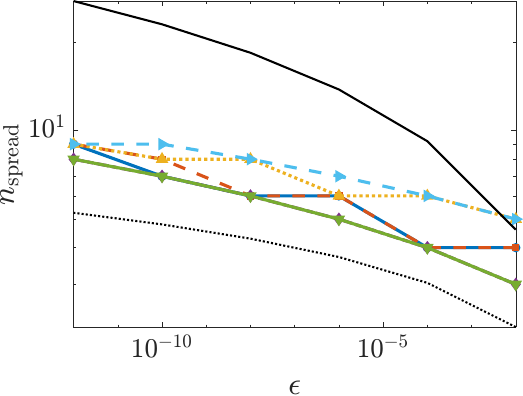}
    \end{minipage}
    
        \vspace{\leggap}
    \includegraphics[width=0.5\textwidth]{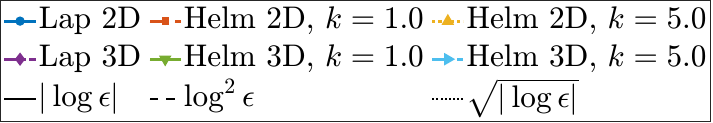}

    \vspace{\capgap}
    \begin{minipage}[t]{\panelw}\centering
        \subcaption{$\nproxy$ versus $\epsilon$}\label{subfig:fig_laplace_helmholtz_nproxy}
    \end{minipage}\hfill
    \begin{minipage}[t]{\panelw}\centering
        \subcaption{$\nspread$ versus $\epsilon$}\label{subfig:laplace_helmholtz_nspread}
    \end{minipage}
    
    \caption{The selected $\nproxy$ (\ref{subfig:fig_laplace_helmholtz_nproxy}) and $\nspread$ (\ref{subfig:laplace_helmholtz_nspread}) for the Laplace and Helmholtz Green's functions in 2D and 3D, as functions of the tolerance $\epsilon$. 
    }
    \label{fig:laplace_helmholtz}
\end{figure}

\begin{figure}
    \centering
    \setlength{\panelw}{0.32\textwidth}
    \begin{minipage}[b]{\panelw}\centering
        \includegraphics[width=\textwidth]{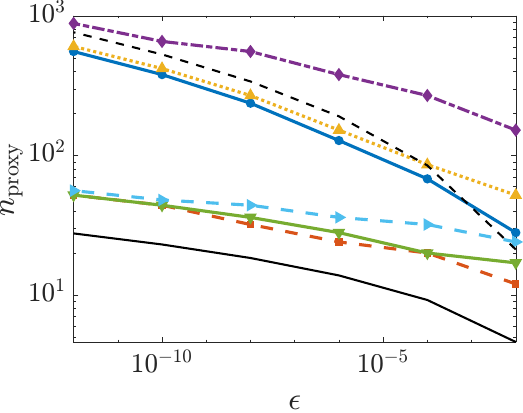}
    \end{minipage}\hfill
    \begin{minipage}[b]{\panelw}\centering
        \includegraphics[width=\textwidth]{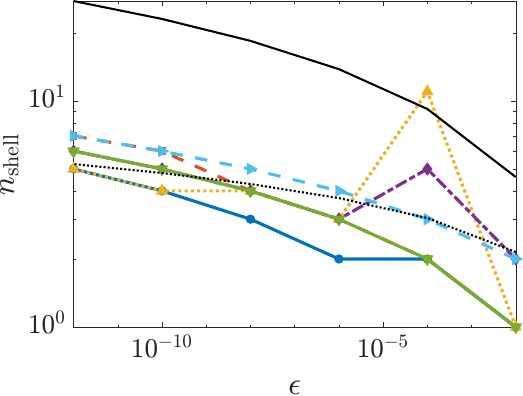}
    \end{minipage}\hfill
    \begin{minipage}[b]{\panelw}\centering
        \includegraphics[width=\textwidth]{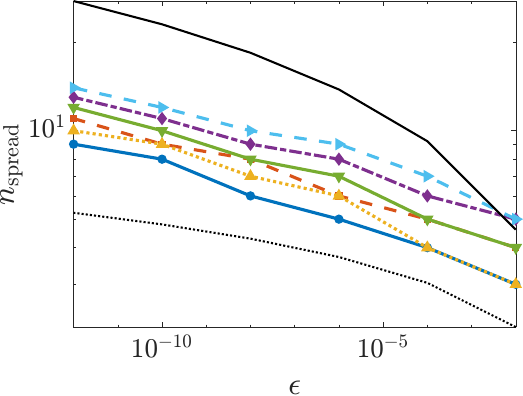}
    \end{minipage}
    
        \vspace{\leggap}
    \includegraphics[width=0.7\textwidth]{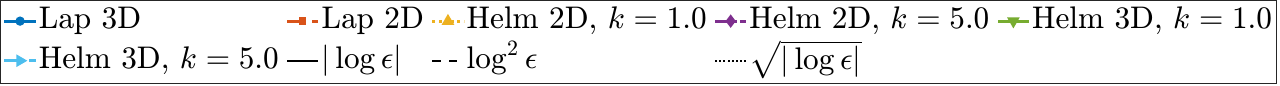}
    
\vspace{\capgap}
    \begin{minipage}[t]{\panelw}\centering
        \subcaption{$\nproxy$ versus $\epsilon$ for the swapped kernel dimension test}\label{subfig:fig_log3d_1overr2d_nproxy}
    \end{minipage}\hfill
    \begin{minipage}[t]{\panelw}\centering
        \subcaption{$\nshell$ versus $\epsilon$ for the swapped kernel dimension test}\label{subfig:fig_log3d_1overr2d_nshell}
    \end{minipage}\hfill
    \begin{minipage}[t]{\panelw}\centering
        \subcaption{$\nspread$ versus $\epsilon$ for the swapped kernel dimension test}\label{subfig:fig_log3d_1overr2d_nspread}
    \end{minipage}

    \caption{The selected $\nproxy$ (\ref{subfig:fig_log3d_1overr2d_nproxy}), $\nshell$ (\ref{subfig:fig_log3d_1overr2d_nshell}), and $\nspread$ (\ref{subfig:fig_log3d_1overr2d_nspread}) for the Laplace and Helmholtz Green's functions in 2D and 3D in the swapped dimension, as functions of the tolerance $\epsilon$. For example, the figure legend ``Lap 2D'' indicates $K(\bx - \by) = \frac{1}{\| \bx - \by \|}$ for $\bx, \by \in \bbR^2$.}
    \label{fig:log3d_1overr2d}
\end{figure}
 
Having tested the algorithm's performance on the Laplace and Helmholtz Green's functions, we now move to a broader class of kernels.
In \cref{fig:matern_struveK} we run an equivalent test for the Mat\'ern kernel in 2D, which is written in terms of the modified Bessel function of the second kind as $C_\nu(r) = \frac{2^{1-\nu}}{\Gamma(\nu)} (2\sqrt{\nu} r)^\nu K_\nu\lp\sqrt{2\nu}r\rp$ and is a standard kernel in geospatial statistics (see \cite{chen_fast_2014} and the references therein) and the Struve function $\mathbf{K}_0(\rho r)$ \cite[eqn. 11.2.5]{NIST:DLMF} in 2D, which occurs in surface wave problems (see \cite{askham2025integralequationsflexuralgravitywaves}). In  \cref{fig:aniso} we also test Gaussian kernels with covariance
\begin{equation*}
    C_{2D}=\frac{1}{16} \begin{pmatrix}
        2 & 1\\1&3
    \end{pmatrix}\quad \text{and}\quad C_{3D}=\frac{1}{16} \begin{pmatrix}
        2 & 1& 1\\1&4&1\\1&1&3
    \end{pmatrix}
\end{equation*}
in 2D and 3D respectively,
and $f(r) = \sin r / (r + r_1^2)$. Both tests show that the number of shells again grows slowly with decreasing tolerance and that $\nproxy$ grows logarithmically. 
\Cref{subfig:accuracy_aniso} shows our method attains high accuracy on these kernels which do not arise from the Green's function of an elliptic PDE.

As our final experiment in this vein, we test our algorithm on higher order equations, namely equations defined by increasing powers of the Laplacian. This time a single shell is sufficient, but we also enforce agreement of some radial derivatives ($n-1$ derivatives for $\Delta^n$) at our proxy points. \Cref{fig:polyharmonic} shows that the higher order equations require a slightly higher $\nproxy$ and $\nspread$. The spike in $\nproxy$ for the smallest $\epsilon$ is caused by the error in the finite difference scheme we used to compute the radial derivatives.

\begin{figure}
    \centering
    \setlength{\panelw}{0.32\textwidth}
    \begin{minipage}[b]{\panelw}\centering
        \includegraphics[width=\textwidth]{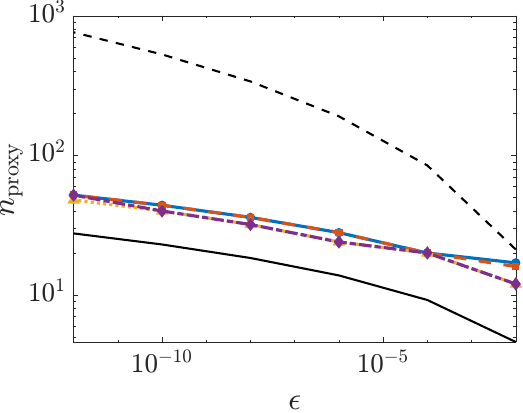}
    \end{minipage}\hfill
    \begin{minipage}[b]{\panelw}\centering
        \includegraphics[width=\textwidth]{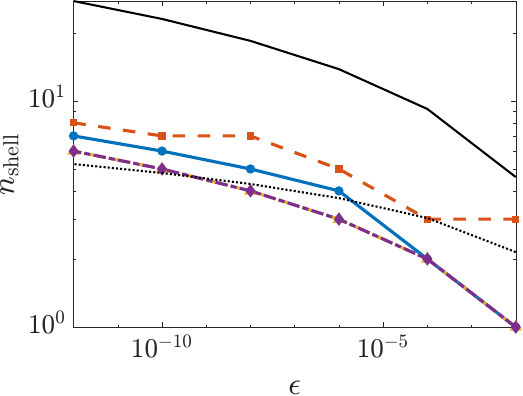}
    \end{minipage}\hfill
    \begin{minipage}[b]{\panelw}\centering
        \includegraphics[width=\textwidth]{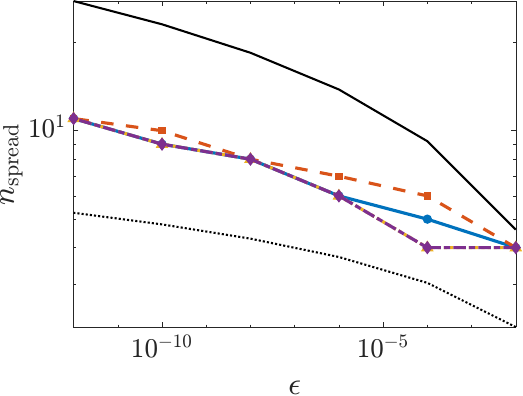}
    \end{minipage}

    \vspace{\leggap}
    \includegraphics[width=0.6\textwidth]{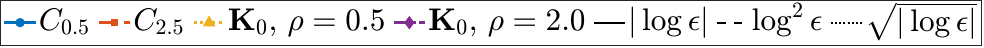}

    \vspace{\capgap}
    \begin{minipage}[t]{\panelw}\centering
        \subcaption{$\nproxy$ versus $\epsilon$}\label{subfig:matern_struveK_nproxy}
    \end{minipage}\hfill
    \begin{minipage}[t]{\panelw}\centering
        \subcaption{$\nshell$ versus $\epsilon$}\label{subfig:matern_struveK_nshell}
    \end{minipage}\hfill
    \begin{minipage}[t]{\panelw}\centering
        \subcaption{$\nspread$ versus $\epsilon$}\label{subfig:matern_struveK_nspread}
    \end{minipage}

\caption{The selected $\nproxy$ (\ref{subfig:matern_struveK_nproxy}), $\nshell$ (\ref{subfig:matern_struveK_nshell}), and $\nspread$ (\ref{subfig:matern_struveK_nshell}) for the Mat\'ern kernel $C_\nu$ and Struve function $\mathbf{K}_0(\rho \,r)$ in 2D, as functions of the tolerance $\epsilon$. }
    \label{fig:matern_struveK}
\end{figure}

\begin{figure}
    \centering
    \setlength{\panelw}{0.32\textwidth}

        \begin{minipage}[b]{\panelw}\centering
        \includegraphics[width=\textwidth]{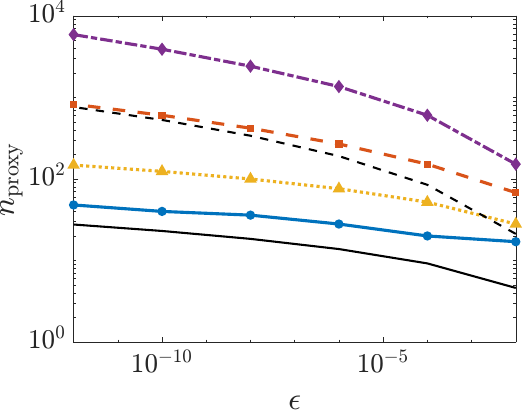}
    \end{minipage}\hfill
    \begin{minipage}[b]{\panelw}\centering
        \includegraphics[width=\textwidth]{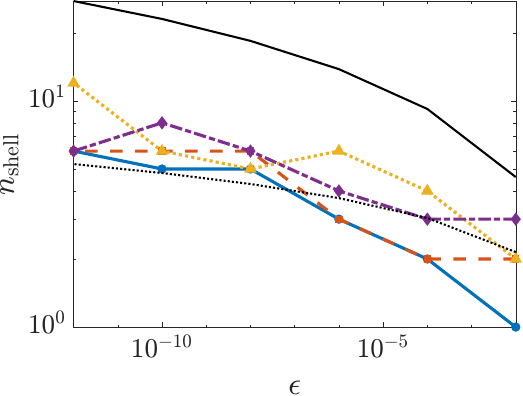}
    \end{minipage}\hfill
    \begin{minipage}[b]{\panelw}\centering
        \includegraphics[width=\textwidth]{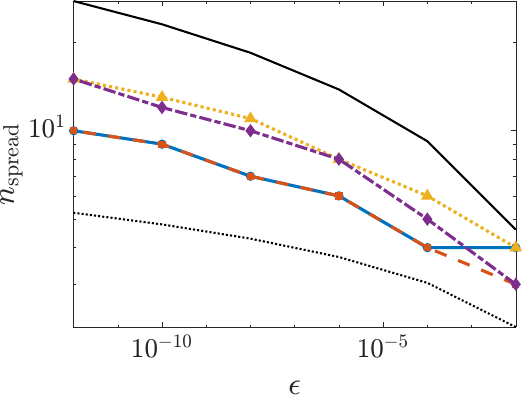}
    \end{minipage}

    \vspace{\leggap}
    \includegraphics[width=0.6\textwidth]{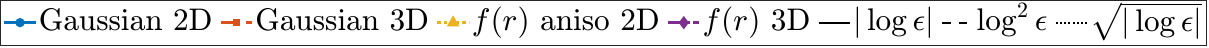}
    
    \vspace{\capgap}
    \begin{minipage}[t]{\panelw}\centering
        \subcaption{$\nproxy$ versus $\epsilon$}\label{subfig:aniso_nproxy}
    \end{minipage}\hfill
    \begin{minipage}[t]{\panelw}\centering
        \subcaption{$\nshell$ versus $\epsilon$}\label{subfig:aniso_nshell}
    \end{minipage}\hfill
    \begin{minipage}[t]{\panelw}\centering
        \subcaption{$\nspread$ versus $\epsilon$}\label{subfig:aniso_nspread}
    \end{minipage}
    
    \caption{The selected $\nproxy$ (\ref{subfig:aniso_nproxy}), $\nshell$ (\ref{subfig:aniso_nshell}), and $\nspread$ (\ref{subfig:aniso_nspread}) for an anisotropic Gaussian and $f(r) = \frac{\sin r}{r+r_1^2}$. \Cref{subfig:accuracy_aniso} shows the accuracy of the $N$-body sum for these examples.}
    \label{fig:aniso}
\end{figure}
\begin{figure}[h!]
\setlength{\panelw}{0.45\textwidth}
    \centering
    \begin{minipage}[b]{\panelw}\centering
        \includegraphics[width=\textwidth]{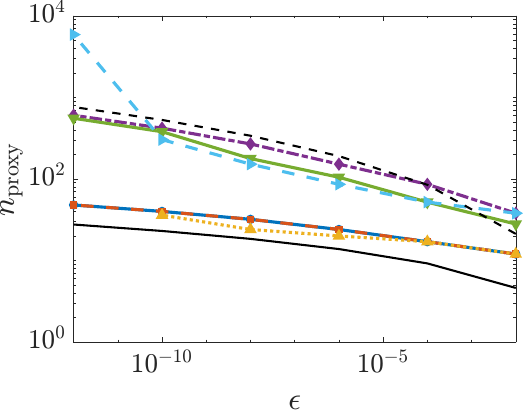}
    \end{minipage}\hfill
    \begin{minipage}[b]{\panelw}\centering
        \includegraphics[width=\textwidth]{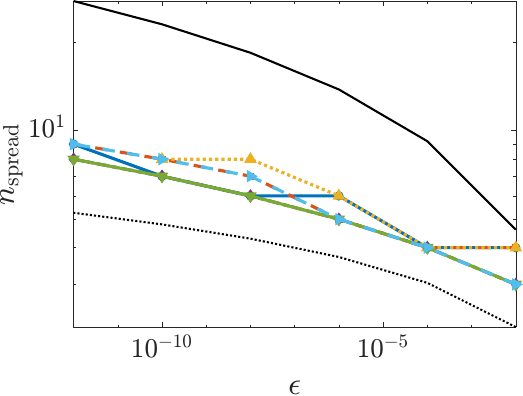}
    \end{minipage}
    
        \vspace{\leggap}
    \includegraphics[width=0.5\textwidth]{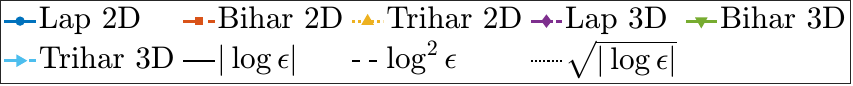}

    \vspace{\capgap}
    \begin{minipage}[t]{\panelw}\centering
        \subcaption{$\nproxy$ versus $\epsilon$}\label{subfig:polyharmonic_nproxy}
    \end{minipage}\hfill
    \begin{minipage}[t]{\panelw}\centering
        \subcaption{$\nspread$ versus $\epsilon$}\label{subfig:polyharmonic_nspread}
    \end{minipage}
    \caption{The selected $\nproxy$ (\ref{subfig:polyharmonic_nproxy}) and $\nspread$ (\ref{subfig:polyharmonic_nspread}) for the Green's functions for the Laplace, biharmonic, and triharmonic problems in 2D and 3D, as functions of the tolerance $\epsilon$. Because these kernels all arise from Green's functions of elliptic PDEs, we use $\nshell=1$.}
    \label{fig:polyharmonic}
\end{figure}

\subsection{Derivatives of the kernel}
\label{sec:der}
Many $N$-body calculations involve source or target derivatives of the kernel. 
Indeed, almost all boundary integral equation methods require them. 
Part of the power of the proxy point method, is that when approximation \cref{eq:approx_kern} holds for the original kernel $K(\bx, \by)$, we can expect a similar approximation to hold for the  derivative of the kernel $K'(\bx, \by ) = \bn(\by) \cdot \nabla_{\by} K(\bx, \by)$.
More formally, let $\cW_k$ be the set of points $\bx$ that are at least $\rho_1$ away from the center of spreading bin $\cB_k$. Then the set of functions $f$ such that
\begin{equation*}
    \opnm{dist}_{C_b(\cW_k)}(f, \opnm{span}\{K(\cdot - \bz_i) \, |\, \bz_i\in \overline{\cB}_k\}) \leq \epsilon
\end{equation*}
is a closed subset of $C_b(\cW_k)$ (the set of bounded continuous functions) in the sup norm.
Thus, if $K$ is decays at infinity and is smooth enough that 
\begin{align*}
    \lim_{h \to 0} \frac{K(\bx - (\by + h \bn(\by))) - K(\bx - \by)}{h}
\end{align*}
converges to $\bn(\by)\cdot \nabla_{\by}K(\bx-\by)$ in $C_b(\cW_k)$, then its derivative will also admit the low rank factorization. This result gives us the following theorem.
\begin{theorem}
    If $K$ is a kernel that is bounded outside the ball $B$ and admits a low rank factorization of the form \cref{eq:approx_kern} in $C_b(\cW_k)$, then for all $\epsilon>0$ there exists a spreading grid such that 
\begin{align}\label{eq:approx_kern_der}
    \sup_{\bx\in \cW_k}\left|\bn(\by)\cdot \nabla_{\by} K(\bx - \by) - \sum_{\bz_i \in \cN(\by)} K(\bx - \bz_i ) \tilde{\bs{\chi}}_i (\by)\right| <\epsilon
\end{align}
        for some weights $\tilde{\bs{\chi}}_i (\by)$, for all $\by\in \cB_k$. Similar expressions hold for other derivatives of $K$.
\end{theorem}
For kernels that grow at infinity, the same argument works if we weight the norm of $C(\cW_k)$ appropriately.
Numerically, we first determine the number of proxy points required without any derivatives using \cref{alg:nspread_nproxy}, then find the weights $\tilde{\bs\chi}_i$ by enforcing \cref{eq:approx_kern_der} at those proxy points. After they are known, we can compute $N$-body sums using a very slight modification of \cref{alg:compute_correction} and \cref{eq:pcfft_approx}.

To test the effect of these derivatives on the accuracy of our algorithm, we run two experiments. In both, we measure the relative $\ell_\infty$ error in the $N$-body calculation for $2{,}000$ sources and $2{,}017$ targets with a uniform random distribution in the unit box. In the first experiment (\cref{subfig:accuracy_lap}), we use our algorithm to compute $N$-body sums for derivatives of the Laplace Green's function, $G_L$,
\begin{gather*}
    k_S(\bx,\by)=G_L(\bx-\by),\quad k_D(\bx,\by)=-\bn(\by)\cdot \nabla G_L(\bx-\by),\\ 
    \text{and}\quad k_{D'}(\bx,\by)= -\bn(\bx)\cdot \nabla (\bn(\by)\cdot \nabla G_L(\bx-\by)).
\end{gather*}
We see that the derivatives can decrease the accuracy, likely because the spread charges are larger, but the error still scales with the tolerance. 

As a second experiment, we explore the accuracy in the kernels that arise when using a standard integral equation method for the flexural wave equation with clamped plate boundary conditions \cite{nekrasov2025boundary,lindsay2018boundary,farkas1989mathematical}. If $\Omega$ is a compact subset of $\bbR^2$ with a smooth boundary, then these equations can be stated as
\begin{align}
\left\{
    \begin{aligned}
    (\Delta^2-k^4)u&=0 &&\quad \text{in } \bbR^2\setminus\Omega,\\
    u&=g_1 && \quad \text{on } \partial\Omega,\\
    \partial_{\bn} u&=g_2 && \quad \text{on } \partial\Omega,\\
\end{aligned}
\right.
\label{eq:biharmonic}
\end{align}
where $\bn$ is the unit normal to $\partial\Omega$, and $g_1$ and $g_2$ are Dirichlet and Neumann boundary data, respectively. The integral equation is built out of the flexural Green's function, which satisfies
\begin{equation*}
    (\Delta^2-k^4)G_F=\delta.
\end{equation*}
The kernels that occur in the integral equation formulation are
\begin{gather*}
    k_{\opnm{BC}}=\begin{bmatrix}
       G_F\\
        \partial_{\bn(\bx)} G_F
    \end{bmatrix},\quad k_{\opnm{rep}}=\begin{bmatrix}
\partial_{\bn(\by)}^3G_F+3\partial_{\bn(\by)}\partial_{\bs{\tau}(\by)}^2G_F&
        -\partial_{\bn(\by)}^2G_F+\partial_{\bs{\tau}(\by)}^2G_F
    \end{bmatrix},\\
    \text{and} \quad k_{\opnm{BIE}}=\begin{bmatrix}
\partial_{\bn(\by)}^3G_F+3\partial_{\bn(\by)}\partial_{\bs{\tau}(\by)}^2G_F&
        -\partial_{\bn(\by)}^2G_F+\partial_{\bs{\tau}(\by)}^2G_F\\
        \partial_{\bn(\bx)}\partial_{\bn(\by)}^3G_F+3\partial_{\bn(\bx)}\partial_{\bn(\by)}\partial_{\bs{\tau}(\by)}^2G_F&
        -\partial_{\bn(\bx)}\partial_{\bn(\by)}^2G_F+\partial_{\bn(\bx)}\partial_{\bs{\tau}(\by)}^2G_F
    \end{bmatrix},
\end{gather*}
where $\bs\tau$ is the unit tangent to $\partial\Omega$. In \cref{subfig:accuracy_flex2d} we see that we achieve accuracy that scales with the tolerance for all of these combinations of derivatives, though the accuracy is not always below the tolerance.
\begin{figure}[h!]
    \centering
    \begin{subfigure}[b]{0.32\textwidth}
        \centering
        \includegraphics[width=\textwidth]{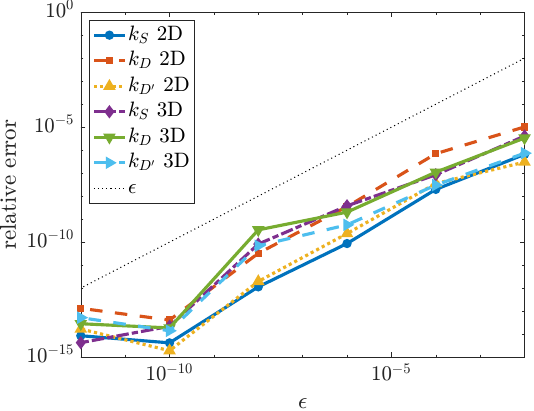}
        \caption{Laplace kernels}
        \label{subfig:accuracy_lap}
    \end{subfigure}
    \hfill
    \begin{subfigure}[b]{0.32\textwidth}
        \centering
        \includegraphics[width=\textwidth]{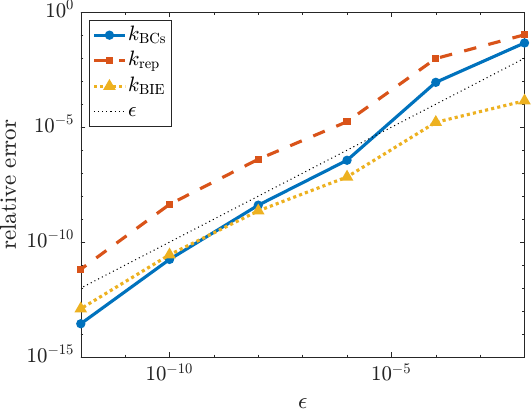}
        \caption{Clamped plate flexural problems}
        \label{subfig:accuracy_flex2d}
    \end{subfigure}
        \begin{subfigure}[b]{0.32\textwidth}
        \centering
        \includegraphics[width=\textwidth]{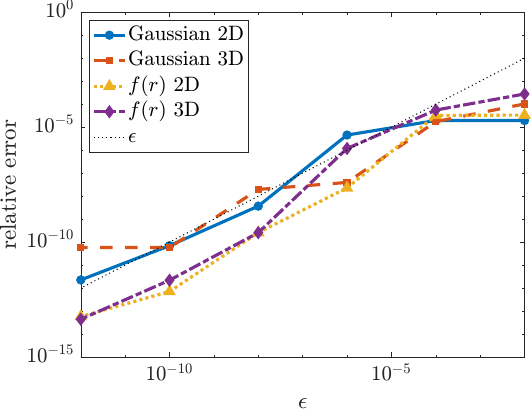}
        \caption{Gaussian and $f(r)=\frac{\sin r}{r+r_1^2}$}
        \label{subfig:accuracy_aniso}
    \end{subfigure}
    \caption{Relative $\ell_\infty$ error in the $N$-body sum versus the proxy tolerance $\epsilon$ for derivatives of the Green's function for Laplace's equation in 2D and 3D (\ref{subfig:accuracy_lap}), in a representation for the flexural boundary value problem with clamped plate boundary conditions and wavenumber $k=10$ (\ref{subfig:accuracy_flex2d}), and the anistropic examples (\ref{subfig:accuracy_aniso}).}
    \label{fig:accuracy}
\end{figure}

\section{Complexity}\label{sec:comp}

Here we will discuss the complexity of the precorrected-FFT method as it scales with $\npts$. In the previous section, we saw that the spreading parameters $(\nspread, \nproxy, \nshell)$ scale at most polylogarithmically with $\epsilon$. These spreading parameters are purely a function of $\epsilon$ and the kernel with no dependence on $\npts$. Here we analyze the complexity for fixed settings of the user inputs $\epsilon, \nnbr$, and $K(\cdot)$.

The complexity of executing \cref{alg:pcfft_apply} is $O(\npts \log\npts)$. This can be seen because $\bA$ has $O(\npts \nnbr)$ entries, the regular grid has $O(\npts \nspread^d / \nnbr)$ grid points, and each sparse spreading operator has $O(\npts \nspread^d )$ nonzero entries. 
The dominant cost in \cref{alg:pcfft_apply} is therefore the FFT/IFFT pair, so the overall complexity scales as $O(\npts \log \npts)$.
The following theorem shows that the precomputation steps have complexity $O(\npts \log \npts)$ as well.

\begin{theorem}[\cite{phillips2002precorrected}]\label{thm:comp}
    Suppose $\nnbr, \nspread, \nproxy,$ and $\nshell$ are fixed.
    For points spread uniformly in a compact subset of $\bbR^d$, the pre-computation cost of before we can run \Cref{alg:pcfft_apply} is $$O\lp\npts\nspread^d/\nnbr\cdot \log(\npts\nspread^d/\nnbr ) + \npts \nspread^{d}(\nproxy\nshell+\nspread^{d}+\nnbr)\rp.$$
    Further, the cost a single evaluation of the $N$-body calculation using \Cref{alg:pcfft_apply} is $$O(\npts\nspread^d/\nnbr \cdot\log(\npts\nspread^d/\nnbr )+ \npts\nnbr).$$
\end{theorem}
\begin{proof}
    The precomputation involves four mains steps.
    \begin{enumerate}
        \item Determining $\npxy$ and $\nspread$ through Algorithm~\ref{alg:nspread_nproxy}, which is independent of $\npts$, but can have non-trivial cost for generic kernels.
        \item Computing the spreading operators $\bS_{\by}$ and $\bS_{\bx}$ requires solving $\npts$ least-squares problems with the same $\nspread^d\times \npxy$ matrix. Since $\npts$ is large, the dominant cost will be $O(\npts\cdot \nproxy\cdot \nshell\cdot  \nspread^{d})$.
\item
    Computing the near-field corrections will cost $O(\npts\cdot \nnbr)$ to get the correct interactions and $O(\npts\cdot \nspread^{d}\cdot(\nnbr+\nspread^{d}))$ to subtract the incorrect value given by the spreading operators.
\item Since the points are spread quasi-uniformly the number of spreading bins will be $O(\npts/\nnbr)$.
    Evaluating the Fourier transform of the kernel will therefore scale like $$O(\npts\nspread^d/\nnbr \cdot\log(\npts\nspread^d/\nnbr )),$$ as the number of grid points scale with the number of bins and the number of grid points per box, $\nspread^d$.
    \end{enumerate}
   Adding up the costs of these steps gives the result.

   In a single application of \Cref{alg:pcfft_apply}, we apply two sparse matrices with $O(\npts\cdot \nnbr)$ entries, use two FFTs, and apply a diagonal matrix. Summing these costs gives the cost of a single matrix apply listed above.
\end{proof}
The above theorem indicates that the algorithm is fastest when $\nspread^d\approx \nnbr \nshell$, though the optimal choice will depend on the computer architecture and memory movement considerations.
It is easy to see that the memory required to store $\bS_{\by}$ and $\bS_{\bx}$ are $\npts\nspread^d$, the cost to store the corrections is $O(\npts \nnbr)$, and the cost to store the grid variables is $O\lp\npts\nspread^d/\nnbr\rp.$ Adding these up gives the following theorem.
\begin{theorem}
    The memory cost to use \Cref{alg:pcfft_apply} is $O\lp \npts \nspread^d + \npts \nnbr + \npts\nspread^d/\nnbr\rp$.
\end{theorem}

\subsection{Scalability}\label{sec:scalability}

In this section, we measure the performance of our method as $\npts$ grows, and compare the performance with other fast $N$-body summation methods. 

\paragraph{Compute environment} We perform our tests on a MacBook Pro, with an Apple M4 Max processor, 64GB of random-access memory, and MacOS version $26{.}5$. 
For all experiments, we use the Apple Accelerate framework as the default BLAS implementation.
Most experiments are driven from the MATLAB compute environment version R2025b.
The 2D and 3D scaling experiments are conducted using a single thread, and the other experiments and numerical examples are conducted using multi-threading.

\paragraph{2D scaling}
First, we experiment in two spatial dimensions, when $K$ is the single-layer kernel from the freespace Helmholtz equation:
\begin{align*}
    K(\bx, \by) &= \frac{i}{4} H_0^{(1)}\left( k \| \bx - \by \| \right),
\end{align*}
where $H_0^{(1)}$ is the Hankel function of the first kind and order $0$, and the wavenumber $k=50$.
We choose a ``quasi-uniform'' distribution of source and target points by successively adding smooth starfish-shaped curves to the domain $[-10, 10]^2$. Each starfish-shaped curve is discretized with about $1{,}200$ points. We compute \cref{eq:Nbody} with tolerance $\epsilon = 10^{-5}$. 
As we increase the number of source and target points, we measure two notions of runtime. 
First, we measure the time required to compute the full sum \cref{eq:Nbody}. 
Second, we measure the time required to apply a new charge vector $\bsigma$, after any necessary precomputations are finished. 
This notion of runtime is pertinent in an integral equation setting, when \cref{eq:Nbody} must be evaluated for dozens or hundreds of new charge vectors $\bsigma$ over the course of an iterative algorithm.

\Cref{fig:2d_scaling} shows the result of this comparison experiment. We compare the performance of different methods: \texttt{PCFFT}, our pure-MATLAB implementation of our kernel-independent pre-corrected FFT method; \texttt{FMM2D} \cite{FMM2D_2023}, a Fortran code called from MATLAB implementing the fast multipole method for the Helmholtz kernel; and \texttt{IFMM} \cite{ho_flam_2020} a pure-MATLAB implementation the interpolative fast multipole method \cite{Martinsson_Rokhlin_2007}, run with default flags and far-field proxy functions.
Among these methods, \texttt{FMM2D} is the only one designed with \textit{a priori} knowledge of the Helmholtz kernel; the other methods are ``kernel independent''. \texttt{FMM2D}'s low total runtimes observed in \cref{subfig:2d_scaling_total} can be attributed to the fact that this method does not need to perform expensive precomputation steps to compress the kernel, which are required by the other methods.
\Cref{subfig:2d_scaling_apply} shows that once the pre-computation is completed \texttt{PCFFT} is able to apply the charge vector an order of magnitude faster than other methods.
Preliminary experiments were conducted with \texttt{BBFMM2D} \cite{Wang_Ambikasaran_2013}, a C++ implementation of the black-box fast multipole method \cite{Fong_Darve_2009}, but this method was unable to accurately compress the highly-oscillatory Helmholtz kernel for this problem.

\begin{figure*}[t!]
    \centering
    \begin{subfigure}[t]{0.45\textwidth}
        \centering
        \includegraphics[width=\textwidth]{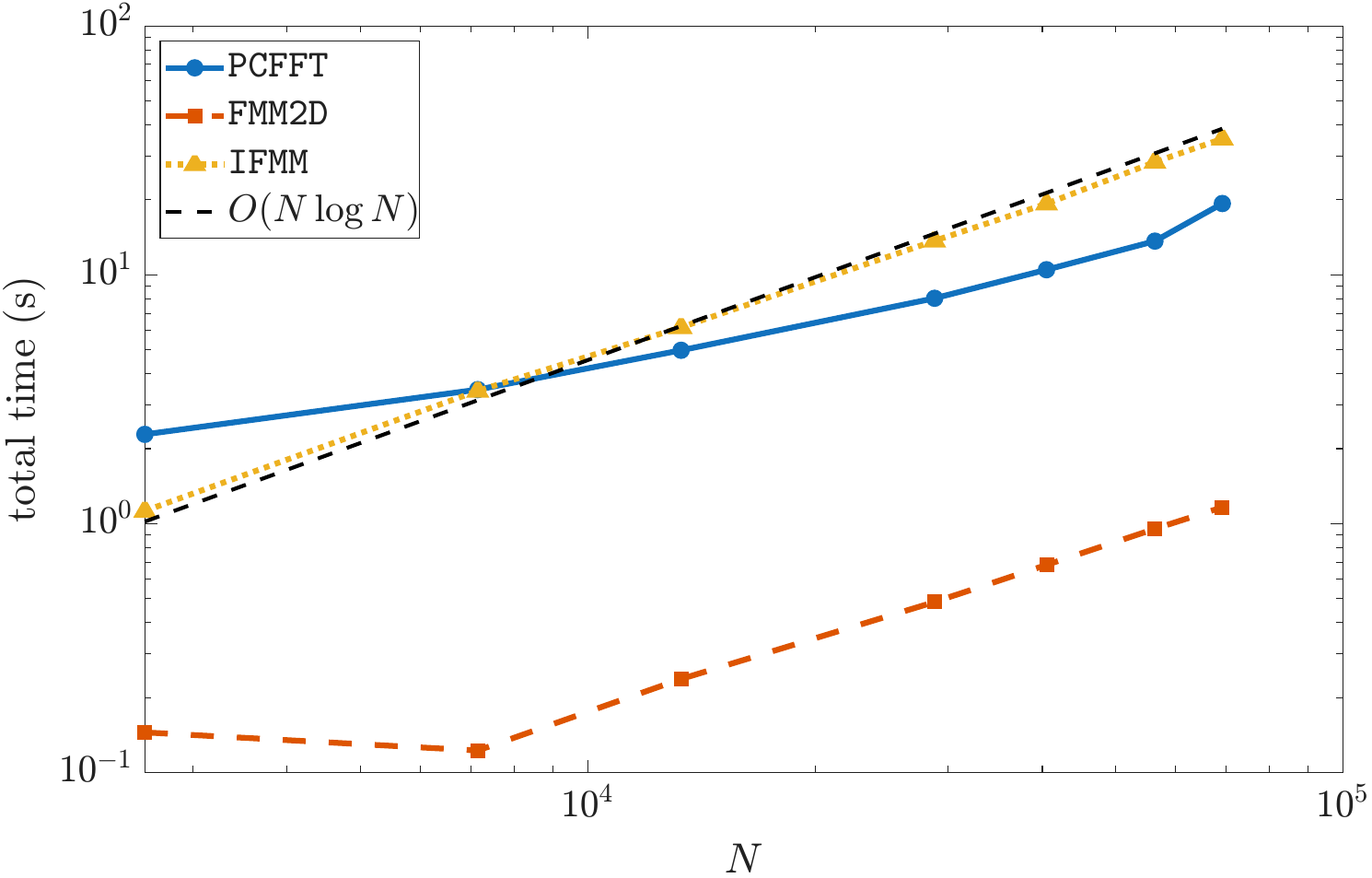}
        \caption{Timing the full computation $\{ \bx, \by, \bsigma \} \to \bu$.}
        \label{subfig:2d_scaling_total}
    \end{subfigure}
    ~
    \begin{subfigure}[t]{0.45\textwidth}
        \centering
        \includegraphics[width=\textwidth]{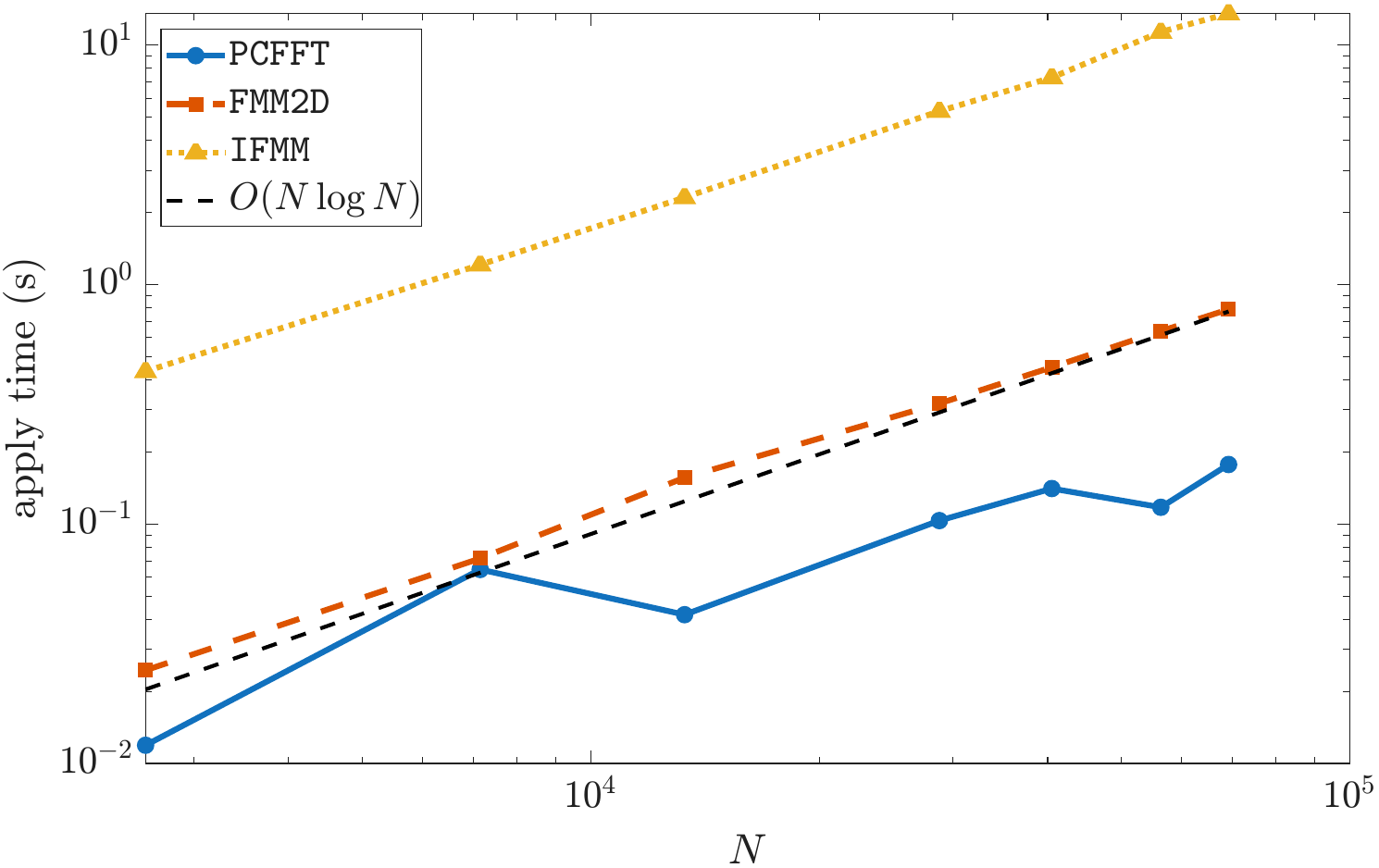}
        \caption{Timing the apply step $\bsigma \to \bu$.}
        \label{subfig:2d_scaling_apply}
    \end{subfigure}
    \caption{Comparing 2D fast summation algorithms computing an $N$-body sum with a Helmholtz kernel and error tolerance $\epsilon = 10^{-5}$. 
    \Cref{subfig:2d_scaling_total} shows the total time to compute the $N$-body sum. 
    The PCFFT method's run time is dominated by the precomputation of the correction matrix. 
    \Cref{subfig:2d_scaling_apply} shows the time spent applying the charge vector $\bs{\sigma}$. 
    PCFFT is much faster in this stage because it relies on optimized FFT codes and sparse linear algebra. }
    \label{fig:2d_scaling}
\end{figure*}

\paragraph{3D scaling}
In three spatial dimensions, we run two similar comparisons. First, we investigate approximating sums with a three-dimensional single-layer Laplace kernel,
\begin{align*}
    K(\bx, \by) &= \frac{1}{4 \pi \|\bx - \by\|},
\end{align*}
with uniformly-sampled source and target point distributions over the box $[-1, 1]^3$. 
This experiment allows us to compare the scaling of our method with other highly-optimized methods designed for fast electrostatic calculations. 
Second, we investigating approximating sums with a three-dimensional single-layer Helmholtz kernel,
\begin{align*}
    K(\bx, \by) &= \frac{e^{ik \|\bx - \by\| }}{i k \|\bx - \by\|},
\end{align*}
with source and target point distributions generated by a triangulation of the surface shown in \cref{fig:var_helmbel}.
We start with a coarse triangulation and generate larger problem sizes by randomly choosing patches to subdivide.
This experiment allows us to measure the scaling of our method in realistic integral equation settings.

We compare the performance of different methods: \texttt{PCFFT} and \texttt{IFMM}, which provide a single interface for both two- and three-dimensional problems; \texttt{FMM3D} \cite{FMM3D_2026}, a Fortran code called from MATLAB implementing the fast multipole method for the Helmholtz kernel; \texttt{ESP} \cite{Greengard_Jiang_Blackwell}, a C++/Fortran code implementing a fast Ewald summation method \cite{liangAcceleratingFastEwald2025}; \texttt{DMK} \cite{Greengard_Jiang_Blackwell}, a C++/Fortran code implementing the dual-space multilevel kernel method \cite{jiangDualspaceMultilevelKernelsplitting2025}; and \texttt{BBFMM3D} \cite{Wang}, a C++ implementation of the black-box fast multipole method.

In \cref{fig:3d_scaling_laplace}, we show the results of the 3D scaling experiment with the Laplace kernel. We see the optimized codes designed specifically for the Laplace kernel are much more performant than the kernel-independent codes, \texttt{PCFFT}, \texttt{IFMM}, and \texttt{BBFMM}. 
Again, \texttt{PCFFT} has the fastest apply step of these algorithms. In \cref{fig:3d_scaling}, we show the results of the 3D scaling experiment with the Helmholtz kernel. 
In this setting, the differences in apply step timings is more stark: after a precomputation step, \texttt{PCFFT} can apply the compressed kernel nearly an order of magnitude faster than the other methods. 
In an integral equation setting, where repeated applications of the kernel operator are required, the cost of the precomputation step will be quickly amortized.
In the 3D Helmholtz scaling test, we performed preliminary experiments with \texttt{BBFMM3D} but found it could not accurately compress this oscillatory kernel.

\begin{figure*}[t!]
    \centering
    \begin{subfigure}[t]{0.45\textwidth}
        \centering
        \includegraphics[width=\textwidth]{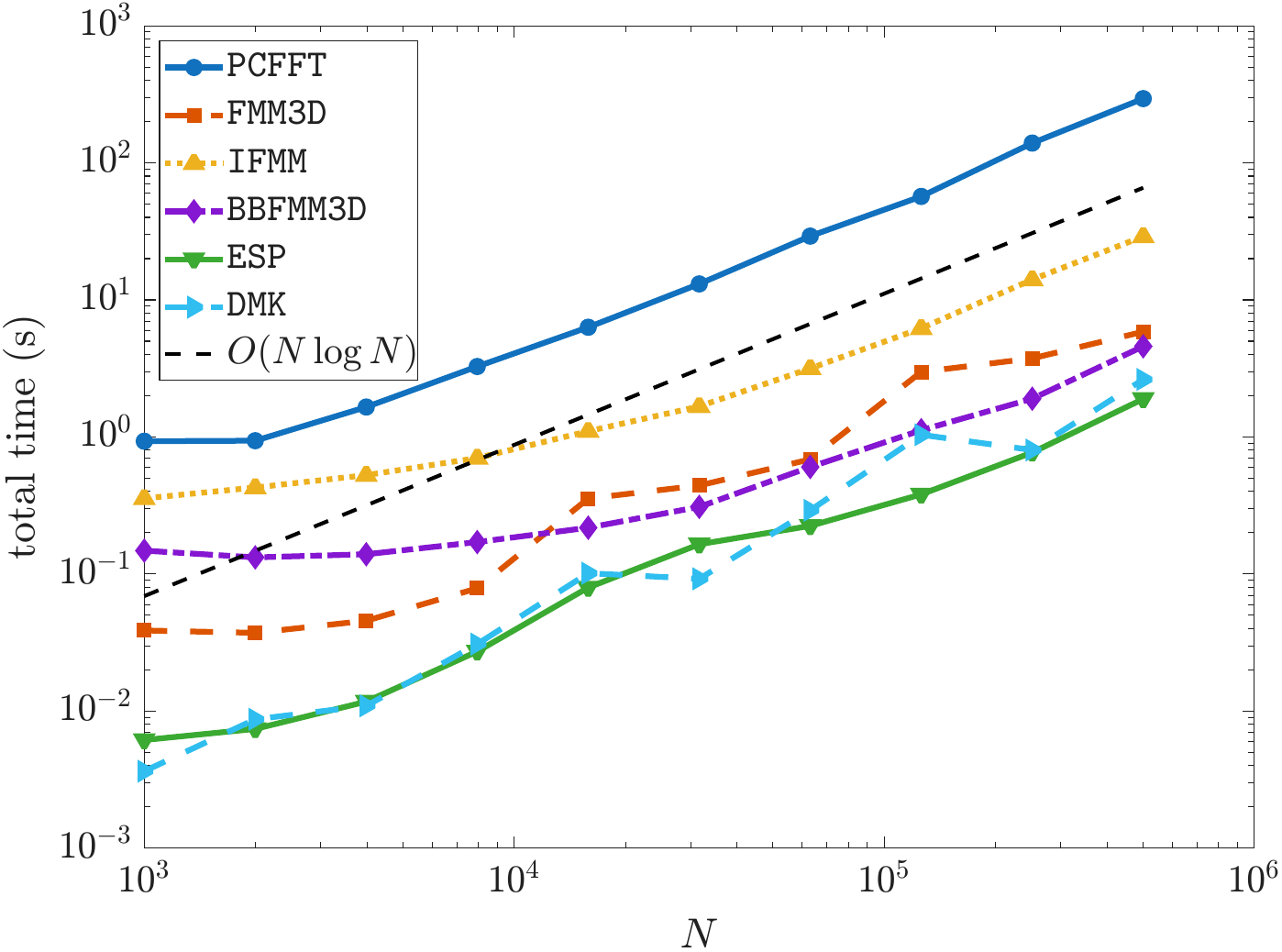}
        \caption{Timing the full computation $\{ \bx, \by, \bsigma \} \to \bu$.}
        \label{subfig:3d_scaling_total_lap}
    \end{subfigure}
    ~
    \begin{subfigure}[t]{0.45\textwidth}
        \centering
        \includegraphics[width=\textwidth]{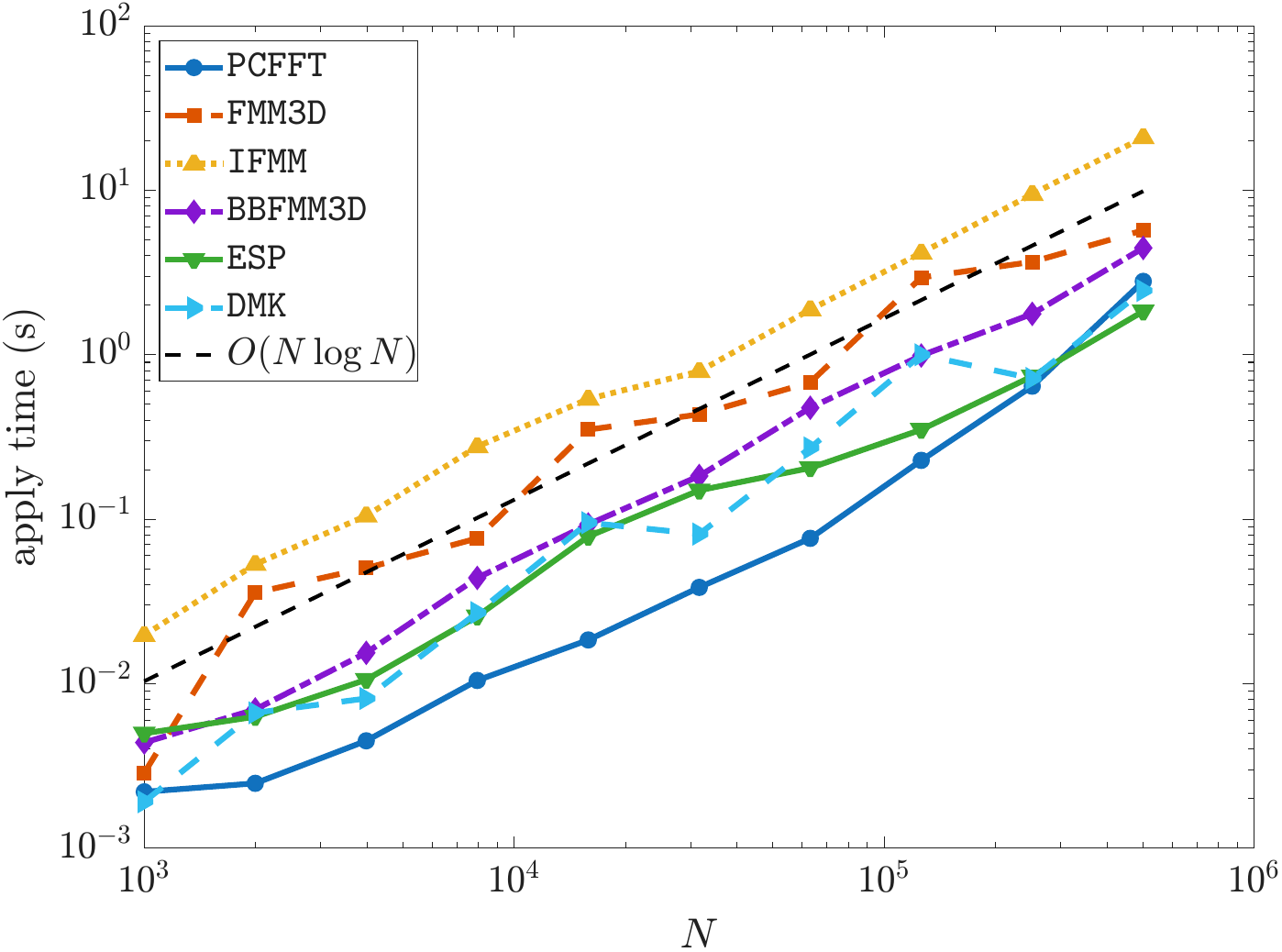}
        \caption{Timing the apply step $\bsigma \to \bu$.}
        \label{subfig:3d_scaling_apply_lap}
    \end{subfigure}
    \caption{Scaling comparison for the 3D Laplace kernel.
    \Cref{subfig:3d_scaling_total_lap} shows the total time to compute the $N$-body sum. 
    \texttt{FMM3D}, \texttt{DMK}, and \texttt{ESP} are all highly-optimized codes specifically designed for the Laplace kernel, and are much faster than the others because they avoid expensive precomputation steps.
    \Cref{subfig:3d_scaling_apply_lap} shows time spent applying the charge vector $\bs{\sigma}$. 
    \texttt{PCFFT} is much faster than the others in this stage because \cref{alg:pcfft_apply} relies on optimized FFT codes and sparse linear algebra. }
    \label{fig:3d_scaling_laplace}
\end{figure*}

\begin{figure*}[t!]
    \centering
    \begin{subfigure}[t]{0.45\textwidth}
        \centering
        \includegraphics[width=\textwidth]{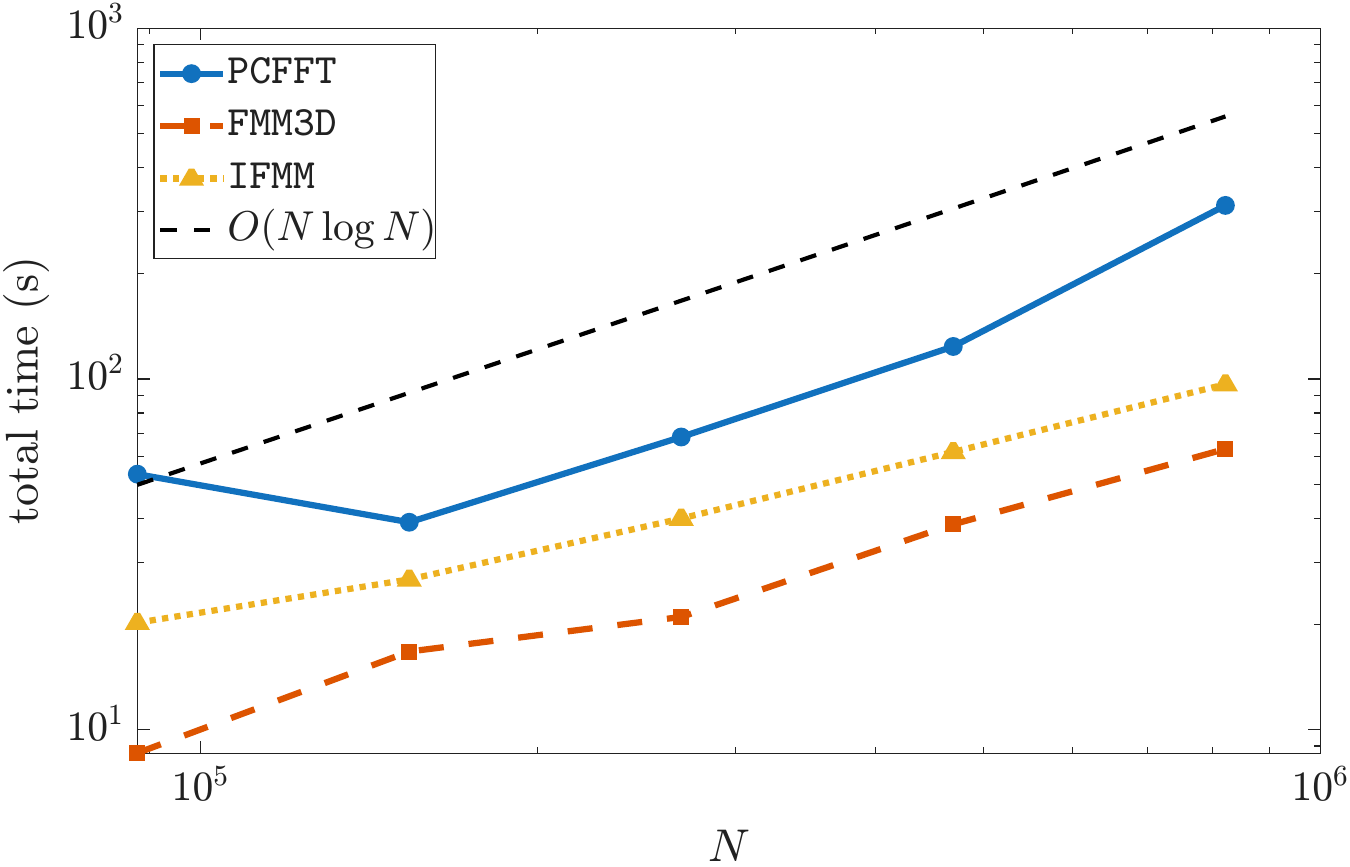}
        \caption{Timing the full computation $\{ \bx, \by, \bsigma \} \to \bu$.}
        \label{subfig:3d_scaling_total}
    \end{subfigure}
    ~
    \begin{subfigure}[t]{0.45\textwidth}
        \centering
        \includegraphics[width=\textwidth]{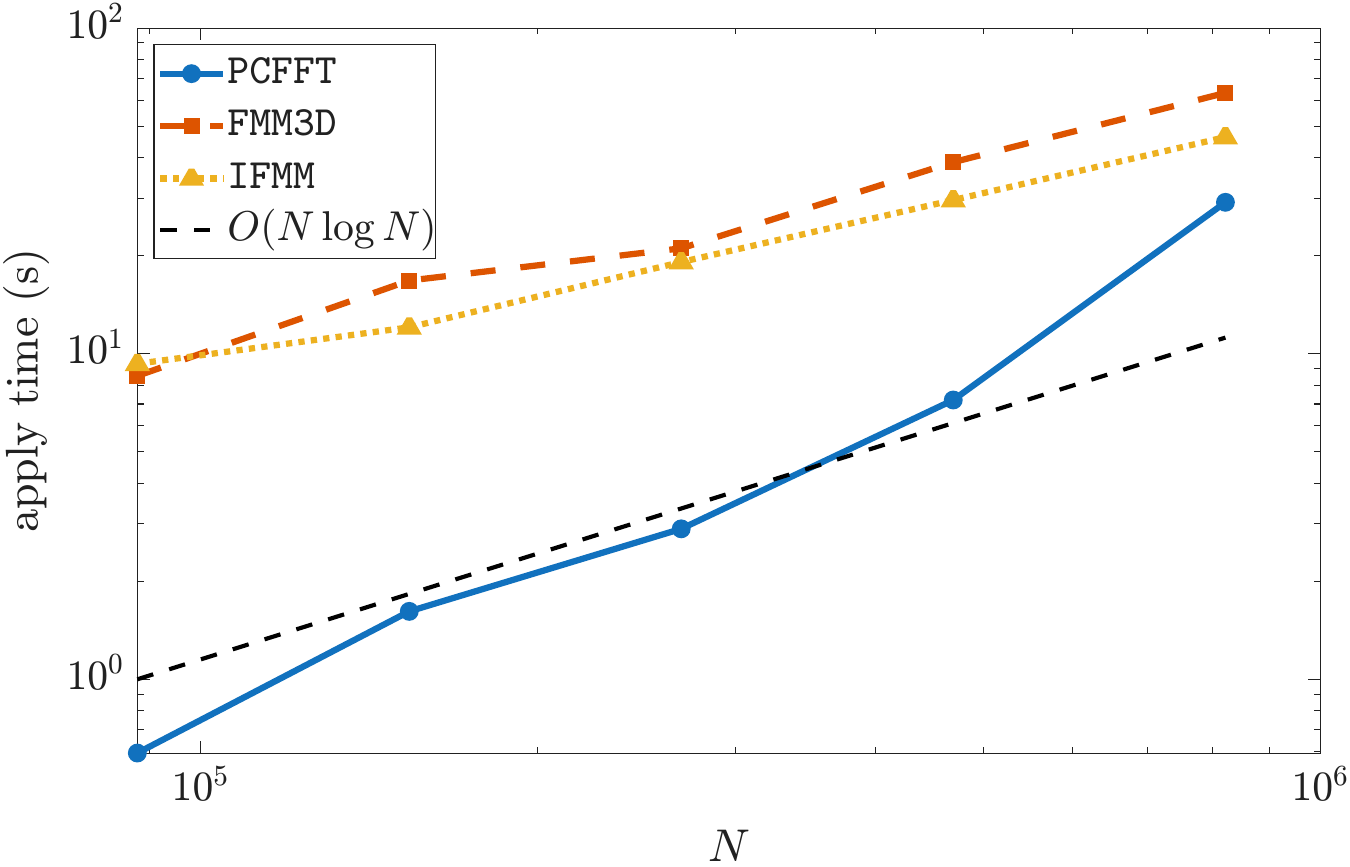}
        \caption{Timing the apply step $\bsigma \to \bu$.}
        \label{subfig:3d_scaling_apply}
    \end{subfigure}
    \caption{Scaling comparison for the 3D Helmholtz kernel. Comparing 3D fast summation algorithms computing an $N$-body sum with a Helmholtz kernel and error tolerance $\epsilon = 10^{-5}$. 
    \Cref{subfig:3d_scaling_total}: total time to compute the $N$-body sum. 
    The PCFFT method's time is dominated by the precomputation of the correction matrix. 
    \Cref{subfig:3d_scaling_apply}: time spent applying the charge vector $\bs{\sigma}$. 
    PCFFT is much faster in this stage because it relies on optimized FFT codes and sparse linear algebra. }
    \label{fig:3d_scaling}
\end{figure*}

\paragraph{Boundary integral equation solve}
In this experiment, we verify the claim that our method is able to improve the runtime of algorithms solving boundary integral equations. 
We solve a Helmholtz scattering problem with sound-hard scatterers. 
Let $\Omega \subset \bbR^2$ denote the union of scatterers, which have piecewise smooth boundary $\partial \Omega$. 
We use the same starfish-shaped scatterers as in the previous 2D scaling experiment.
We compute a solution $u(x)$ of the following partial differential equation:
\begin{align}\label{eq:helm_PDE}
\left\{ 
    \begin{aligned}
         \Delta u(\bx) + k^2 u(\bx) &= 0, &&\quad \bx \in \bbR^2 \backslash \Omega, \\
        \partial_{\bn} u(\bx) &= g(\bx), &&\quad \bx \in \partial \Omega, \\
        \sqrt{r} \left( \frac{\partial u}{\partial r} - ik u\right) &\to 0, &&\quad r = \|\bx\| \to \infty.
    \end{aligned}
\right.
\end{align}
We solve this problem using the single layer representation 
\begin{equation*}
    u(\bx) = \int_{\partial\Omega} G_k(\bx,\by) \sigma(\by)d\by,
\end{equation*}
where $G_k(\bx,\by) = \frac{i}4H_0^{(1)}(\|\bx-\by\|)$ is the 2D Helmholtz Green's function. This representation solves \cref{eq:helm_PDE} provided $\sigma$ solves the integral equation
\begin{equation}\label{eq:helm_IE}
    -\frac{\sigma(\bx)}{2} + \int_{\partial\Omega} \partial_{\bn(\bx)}G_k(\bx,\by) \sigma(\by)d\by= g(\bx)
\end{equation}
on $\partial\Omega$. 
To generate a variety of problem sizes in this setting, we add more and more starfish-shaped scatterers to the domain. The addition of each scatterer both increases the number of discretization points, and increases the conditioning of the problem as multiple-scattering effects are increased. 

In this experiment, we use both \texttt{PCFFT} and \texttt{FMM2D} to accelerate the application of the boundary integral operator \cref{eq:helm_IE}, and solve the resulting linear system iteratively using GMRES. 
We also measure the runtime of \texttt{RSKELF}, a recursive skeletonization based direct solver which computes a factored inverse of the operator in \cref{eq:helm_IE}. 
\Cref{fig:scattering_comparison} shows the results of this scaling experiment. 
At $\npts = 13{,}168$, the iterative methods require $526$ GMRES iterations and \texttt{PCFFT} is faster than the direct solver. 
At the two largest problem sizes, the direct solver computes the solution faster than both iterative methods, which require $804$ and $916$ GMRES iterations on these problem sizes, respectively.
\begin{figure}
    \centering
    \begin{subfigure}[t]{0.44\textwidth}
    \centering
    \includegraphics[width=\textwidth]{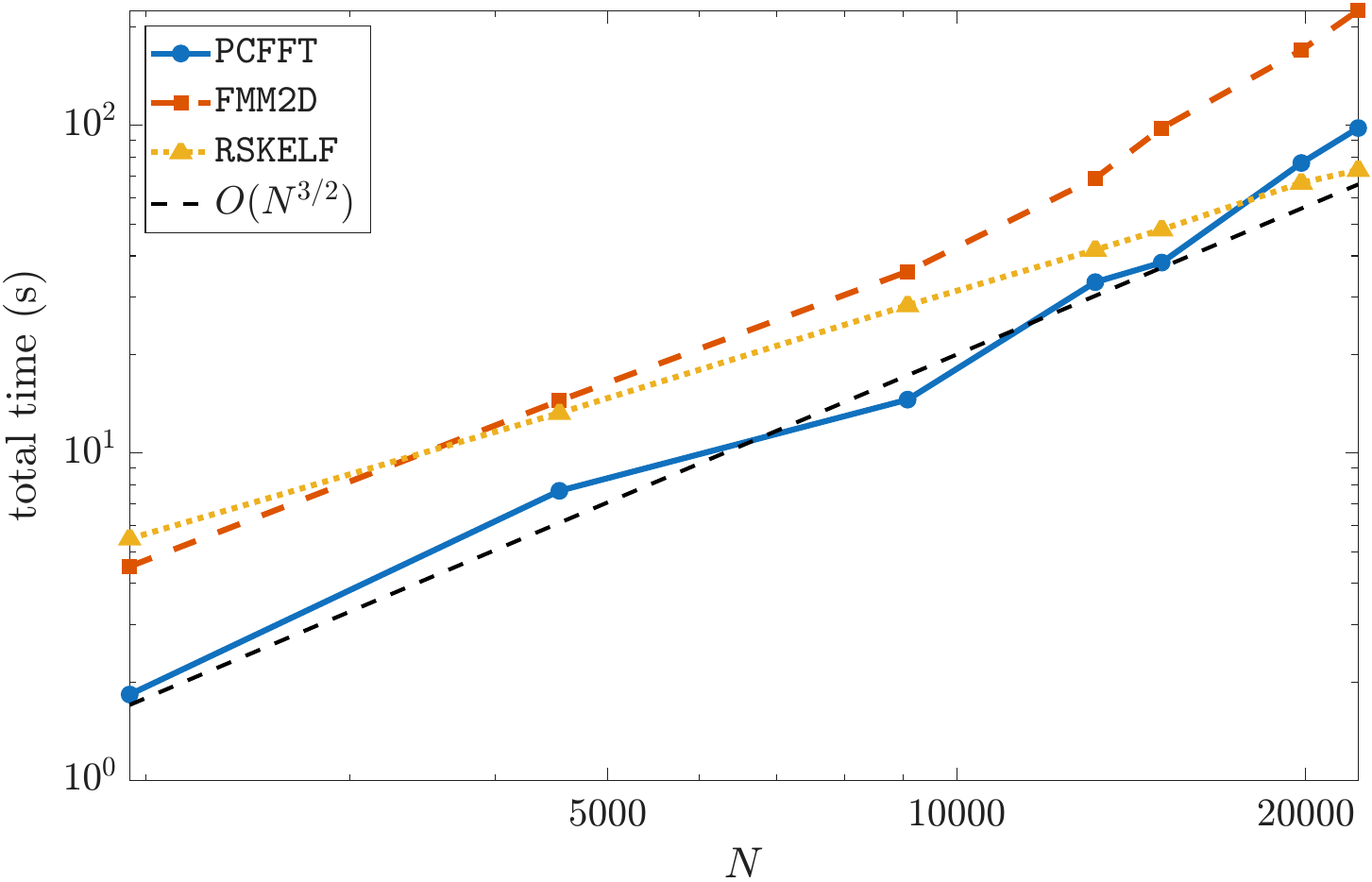}
    \caption{Timing the BIE solve}
    \label{subfig:bie_solve_times}
    \end{subfigure}
    \begin{subfigure}[t]{0.51\textwidth}
    \centering
    \includegraphics[width=\textwidth]{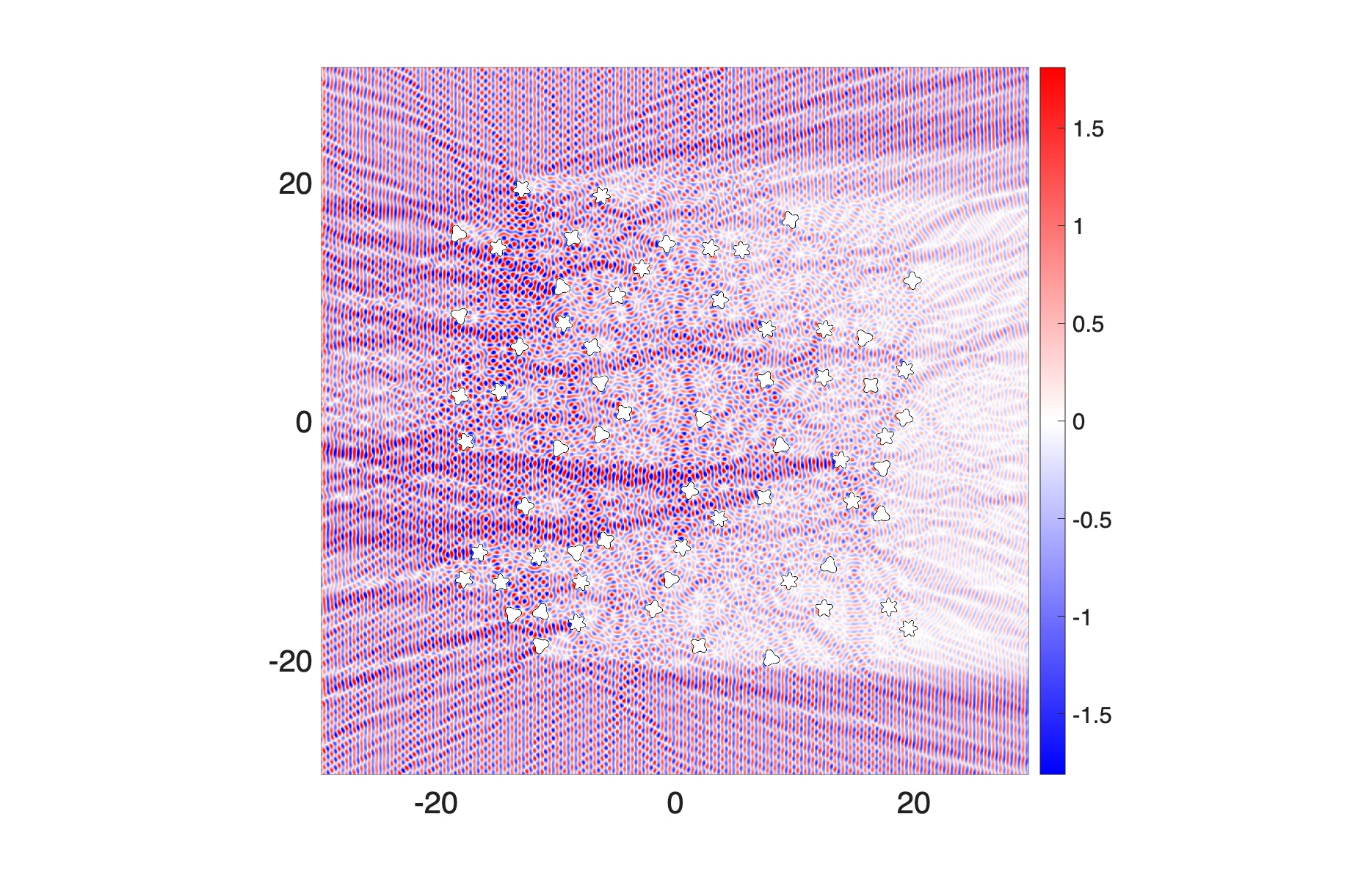}
    \caption{Solution $u$.}
    \label{subfig:soln}
    \end{subfigure}
    \caption{Comparing the solution times of a BIE formulation of a Helmholtz scattering problem. \Cref{subfig:bie_solve_times}: total time to solve the boundary integral equation \cref{eq:helm_IE}. All methods use all 16 cores on a Macbook Pro M4. \Cref{subfig:soln}: visualizing the imaginary part of the solution at the largest problem size. Different problem sizes are generated by adding starfish scatterers to the domain.}
    \label{fig:scattering_comparison}
\end{figure}
\section{Examples}\label{sec:numerics}
We now demonstrate the our algorithm on a few applications, focusing on integral equation settings where Fourier analysis based fast $N$-body algorithms are unavailable.

\subsection{Eigendecomposition of the flexural clamped plate biharmonic problem}
We use the kernels introduced in \cref{sec:der} to find eigenvalues and eigenfunctions of the biharmonic operator with clamped plate boundary conditions. 
That is, for a given 2D domain $\Omega$, we search for pairs $k$ and $u$ which solve \cref{eq:biharmonic}.
We use three domains $\Omega$, corresponding to the shape of the A-23 iceberg, which was pushed into the Weddell Sea and was suddenly broken apart by storms between July and September, 2025. We construct $\Omega_1, \Omega_2$, and $\Omega_3$ by tracing high-resolution satellite imagery of the iceberg in July, August, and September of 2025, respectively \cite{Doris_2025,Doris_2025a,Doris_2025b}. 
This experiment gives us insight to how the eigenfunctions and eigenvalues change with the changing iceberg shape.

A popular class of methods for finding these eigenvalues involve looking for poles of the function 
\begin{align*}
    S(k) = 1/(f^* F_k^{-1} g),
\end{align*}
where $F_k$ is a discrete system matrix of the system at wavenumber $k$. Methods of this form include contour integration methods, such as those developed in \cite{beyn2012integral,colbrook2025avoiding}, and methods that directly look for roots of $S(k)$ \cite{bruno_evaluation_2026,effenberger2012chebyshev} and related functions (see \cite{boyd2002computing,zhao2015robust}). In this example, we use the later method, building a Chebyshev expansion of $S(k)$. For $F_k$ is the boundary integral operator corresponding to the clamped plate biharmonic problem in the left-hand side of \cref{eq:biharmonic}, with wavenumber $k$. The operator
$F_k$ will be the map from auxiliary densities to boundary data $g = (g_1, g_2)$. 
Thus, $F_k^{-1}$ maps from boundary data to auxiliary densities defining the solution. 
If $k$ is an eigenvalue of the operator, the resonance will cause a pole in $S(k)$.

To find the poles of $S(k)$, we set a search interval $k \in [12,16]$, discretize the interval with $80$ Chebyshev points, and evaluate $S(k)$ at these points.
At each value of $k$, we compute $F_k^{-1}g$ by solving the resulting boundary integral equation iteratively using GMRES, where the matrix-vector products are accelerated using \texttt{PCFFT}.
Given samples of $S(k)$ at Chebyshev points, we fit a Chebyshev series for $1/S(k)$ and find the zeros of this function. 
The roots of $1/S(k)$ are poles of $S(k)$ and, equivalently, resonant wavenumbers of \cref{eq:biharmonic}.
Once we have found an eigenvalue $k$, we find an auxiliary density in the nullspace of $F_{k}^{-1}$. 
Because we are at a resonance, we expect to incur a large iteration complexity to solve this problem, so we use the \texttt{RSKELF} direct solver, implemented in FLAM \cite{ho_flam_2020}.

We conduct the search for eigenvalues and eigenfunctions using three iceberg geometry shapes. \Cref{tab:eigenvalues} contains statistics of this experiment.
\Cref{fig:eigenfunctions} shows the recovered eigenfunctions on the three different geometries. Even with a relatively un-optimized algorithm for finding roots of $S(k)$, we are able to solve these nonlinear eigenvalue problems quite efficiently.
\begin{table}[h!]
    \centering
    \begin{tabular}{lllllll}
    \toprule
        Geometry & Eigenvalues & Search interval & $\npts$ & $n_{\text{iter.}}$ & $T_{\text{res.}}$ & $T_{\text{dir.}}$ \\
        \midrule
        $\Omega_1$ & $\{13.02, 14.10, 15.16, 15.97\}$ & $[12, 16]$ & $4{,}928$ & $48.7$ & $11.9$ & $32.9$  \\
        $\Omega_2$ & $\{12.39, 14.89, 15.87 \}$ & $[12, 16]$ & $7{,}104$ & $56.5$ & $17.4$  & $92.6$  \\
        $\Omega_3$ & $\{13.71, 14.33\}$ & $[12, 16]$ & $3{,}280$ & $59.6$ & $9.4$ & $27.5$ \\
        \bottomrule
    \end{tabular}
    \caption{Statistics of the eigendecomposition experiment. 
    For each geometry, the eigenvalues found by our search are shown in the second column.
    The third column shows the interval over which we searched for eigenvalues.
    The remaining columns are $\npts$: number of discretization points along the boundary, $n_{\text{iter.}}$: average number of GMRES iterations needed to solve each clamped plate boundary value problem, $T_{\text{res.}}$: average time (seconds) to solve each bounday integral equations when sampling the resolvent $S(k)$, and $T_{\text{dir.}}$: average time (seconds) required by the direct solver to compute each eigenfunction.
    Visualizations of the eigenfunctions are provided in \cref{fig:eigenfunctions}.
    }
    \label{tab:eigenvalues}
\end{table}

\begin{figure}
    \centering
    \begin{subfigure}[t]{0.19\textwidth}
        \centering
        \includegraphics[width=\textwidth]{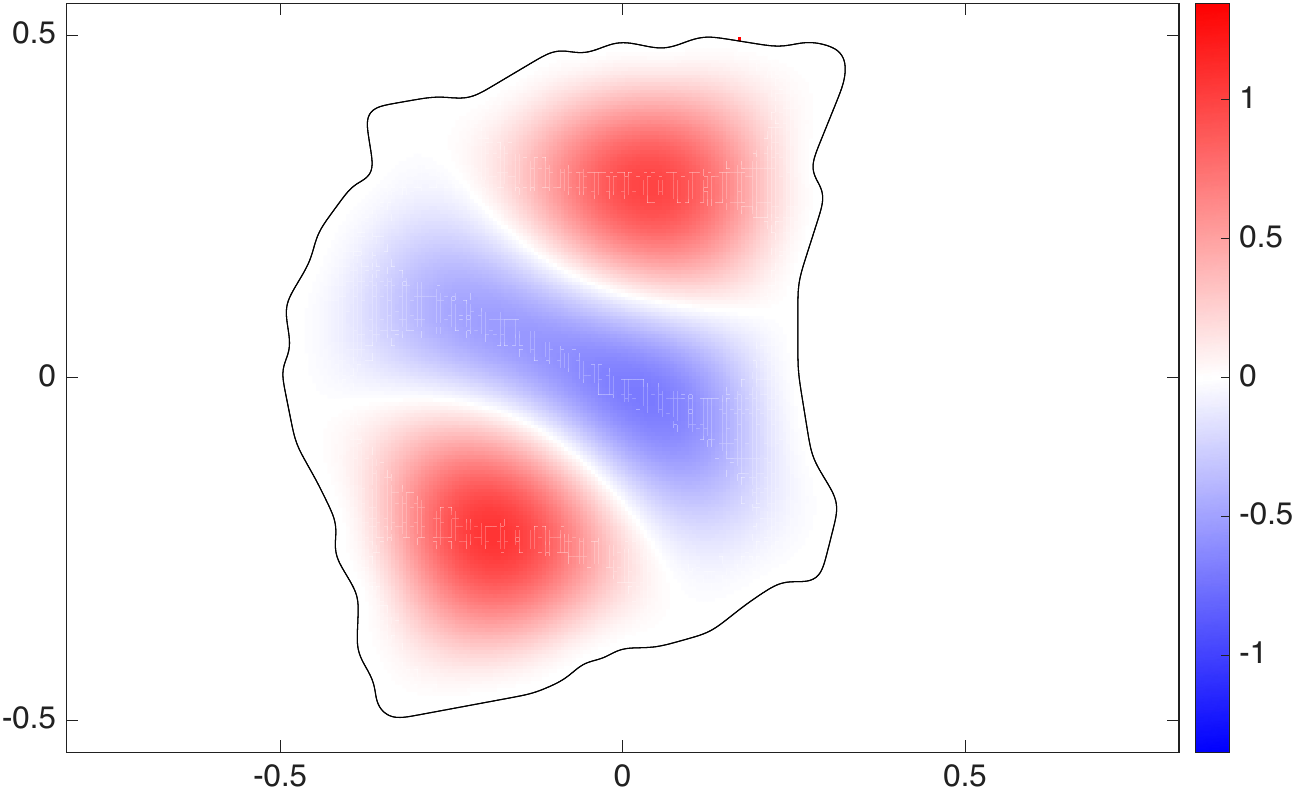}
        \caption{$\Omega_1$, $k=13.02$}
        \label{fig:eigenfunctions_a}
    \end{subfigure}
    \hfill
    \begin{subfigure}[t]{0.19\textwidth}
        \centering
        \includegraphics[width=\textwidth]{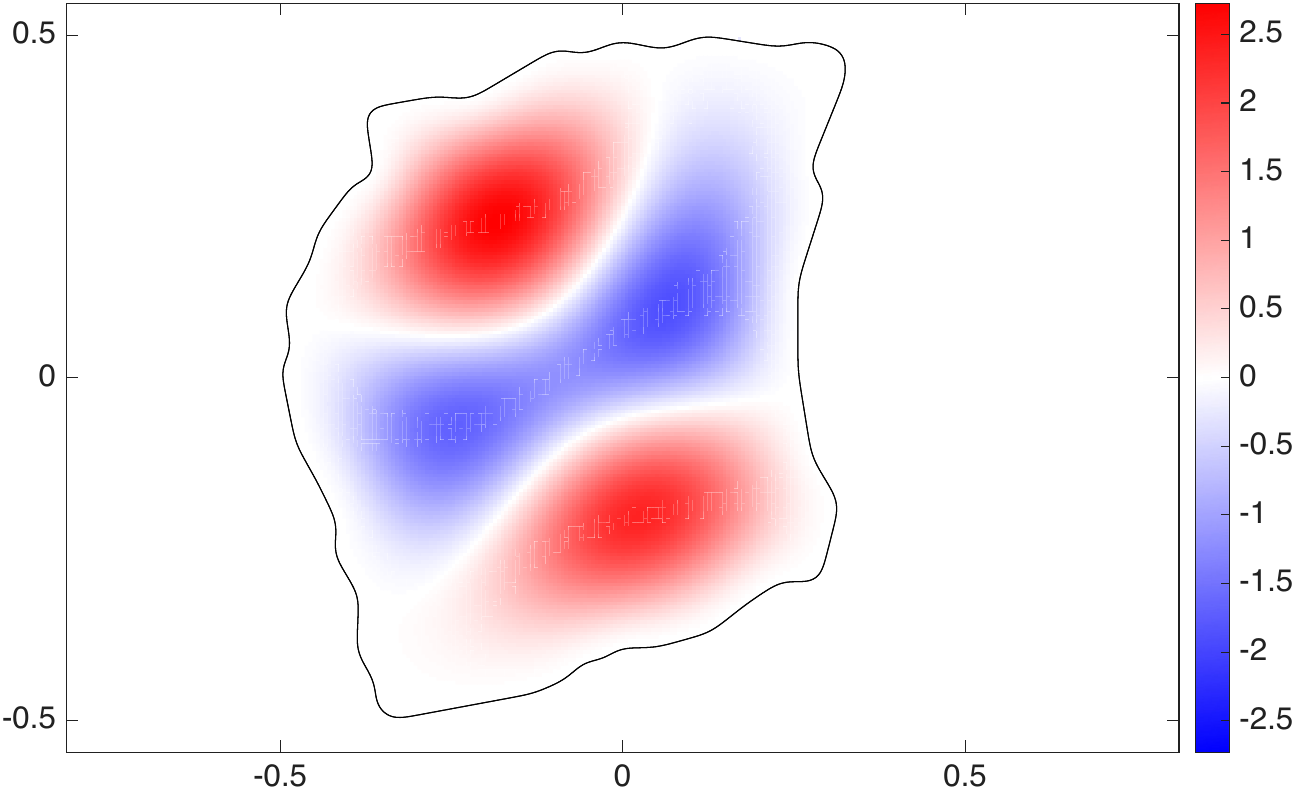}
        \caption{$\Omega_1$, $k=14.10$}
        \label{fig:eigenfunctions_b}
    \end{subfigure}
    \hfill
    \begin{subfigure}[t]{0.19\textwidth}
        \centering
        \includegraphics[width=\textwidth]{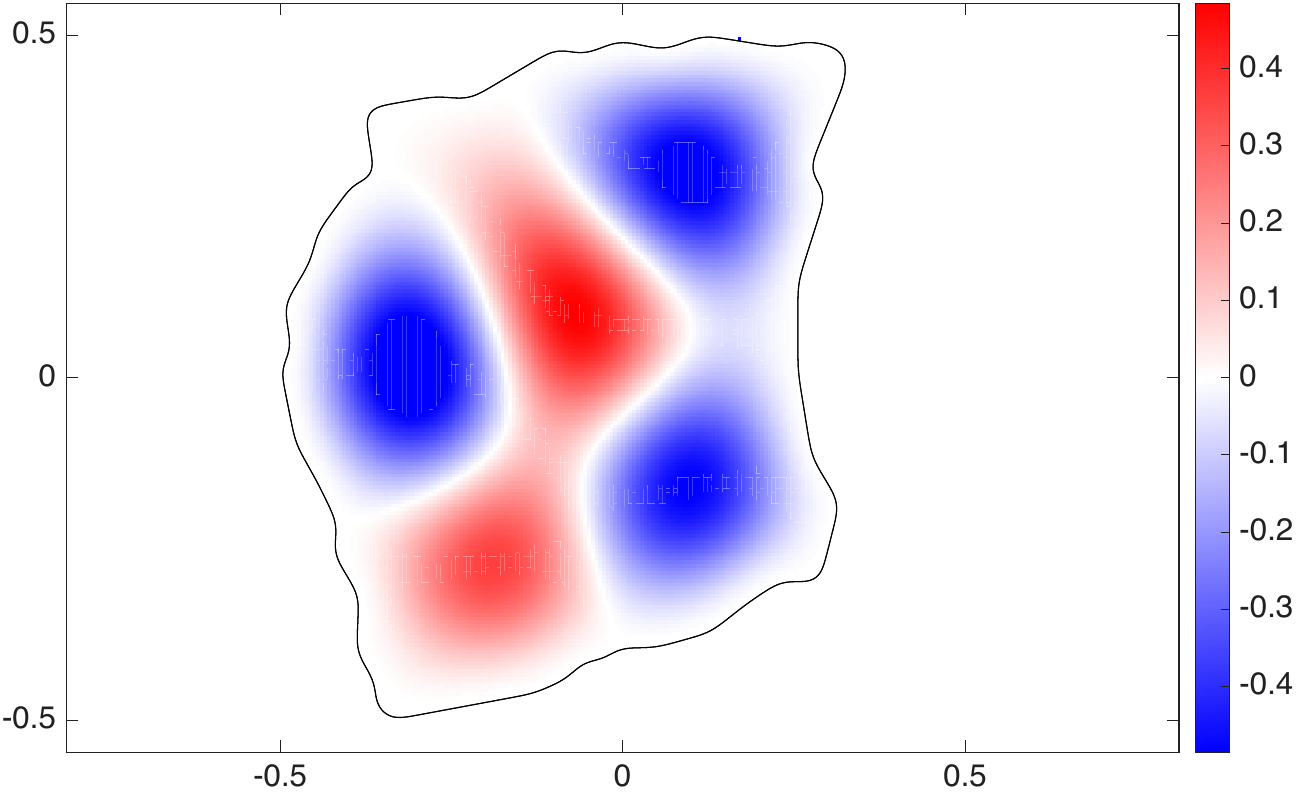}
        \caption{$\Omega_1$, $k=15.16$}
        \label{fig:eigenfunctions_c}
    \end{subfigure}
    \hfill
    \begin{subfigure}[t]{0.19\textwidth}
        \centering
        \includegraphics[width=\textwidth]{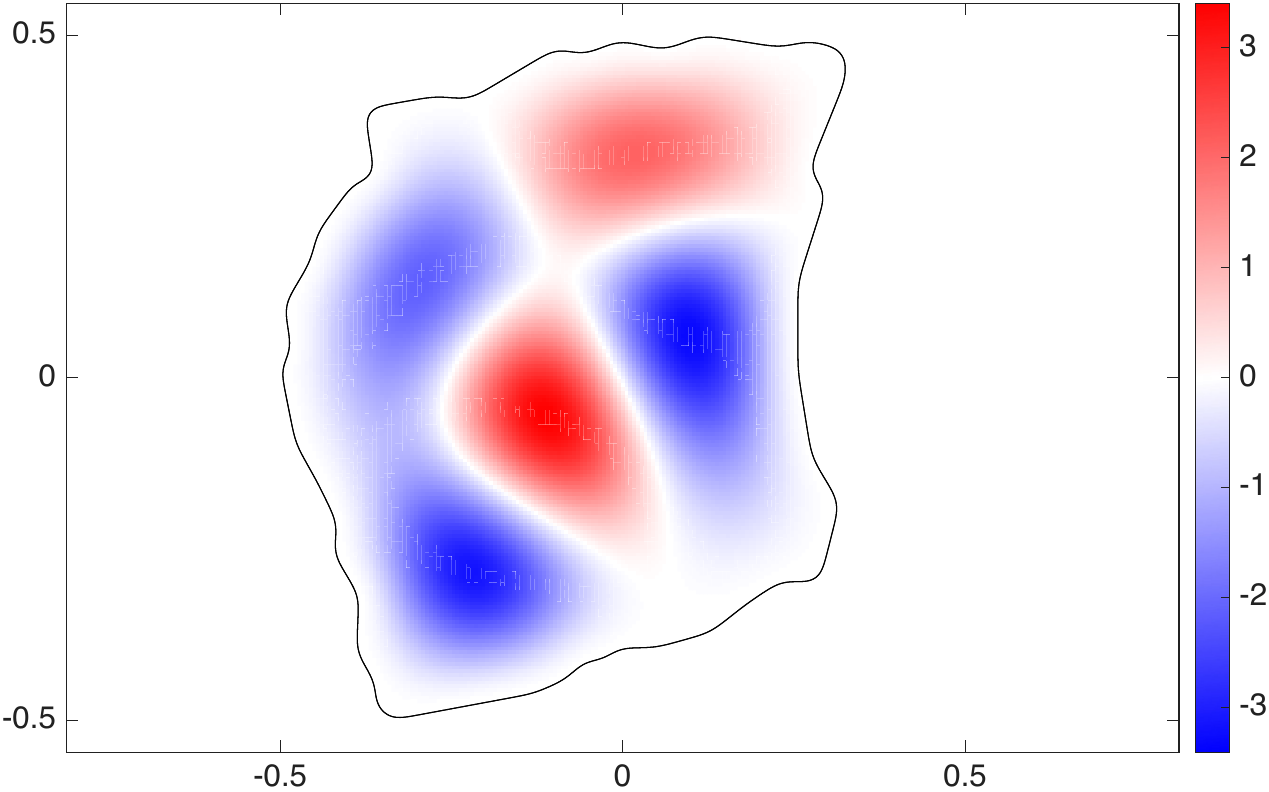}
        \caption{$\Omega_1$, $k=15.97$}
        \label{fig:eigenfunctions_d}
    \end{subfigure}

    \vspace{0.5em}
    \begin{subfigure}[t]{0.19\textwidth}
        \centering
        \includegraphics[width=\textwidth]{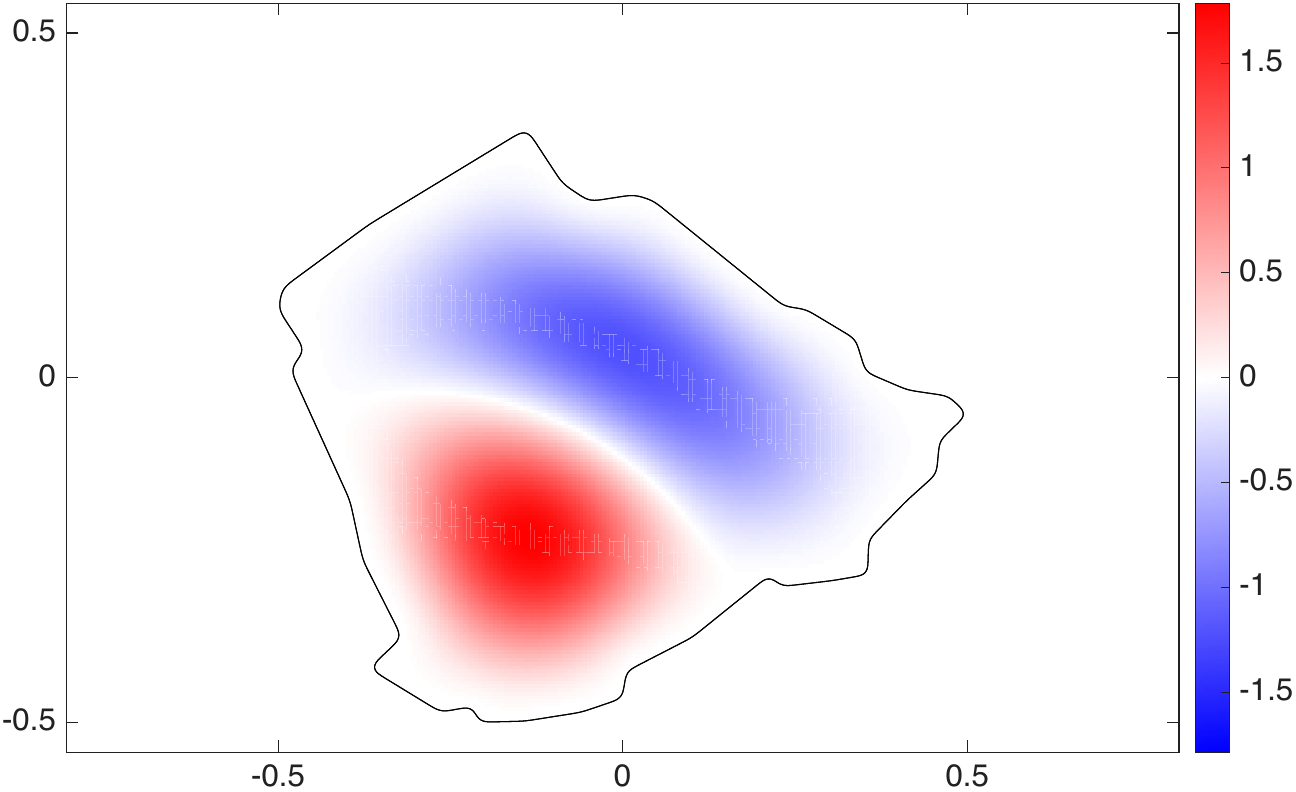}
        \caption{$\Omega_2$, $k=12.39$}
        \label{fig:eigenfunctions_e}
    \end{subfigure}
    \hfill
    \begin{subfigure}[t]{0.19\textwidth}
        \centering
        \includegraphics[width=\textwidth]{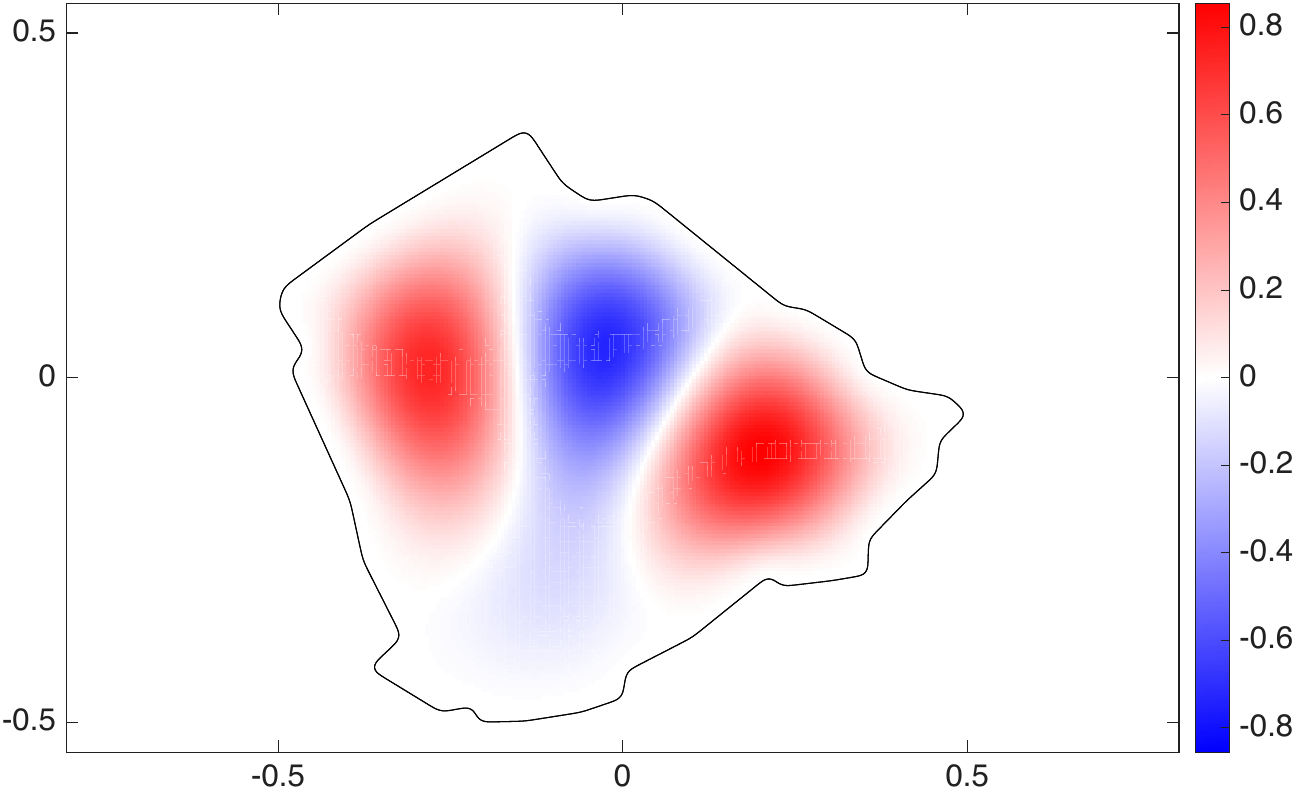}
        \caption{$\Omega_2$, $k=14.89$}
        \label{fig:eigenfunctions_f}
    \end{subfigure}
    \hfill
    \begin{subfigure}[t]{0.19\textwidth}
        \centering
        \includegraphics[width=\textwidth]{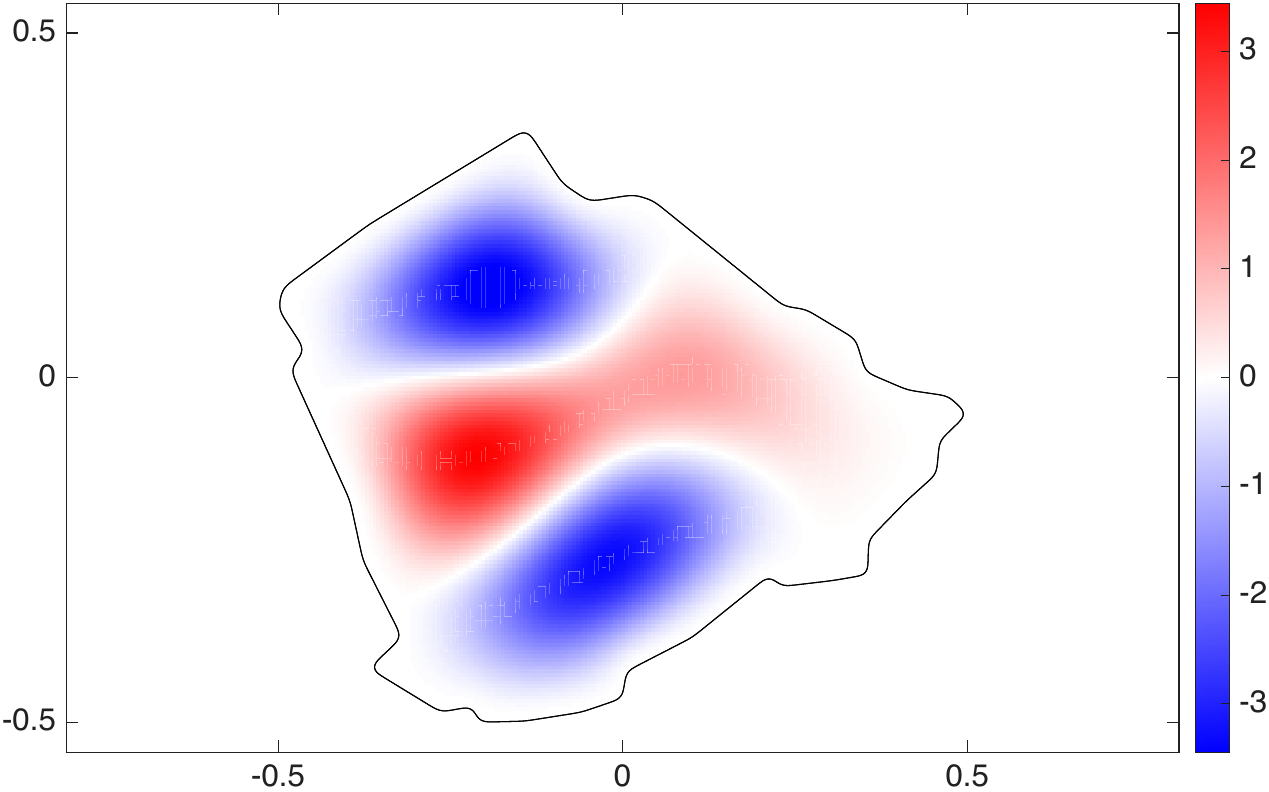}
        \caption{$\Omega_2$, $k=15.87$}
        \label{fig:eigenfunctions_g}
    \end{subfigure}
    \hfill
    \begin{subfigure}[t]{0.19\textwidth}
        \centering
        \includegraphics[width=\textwidth]{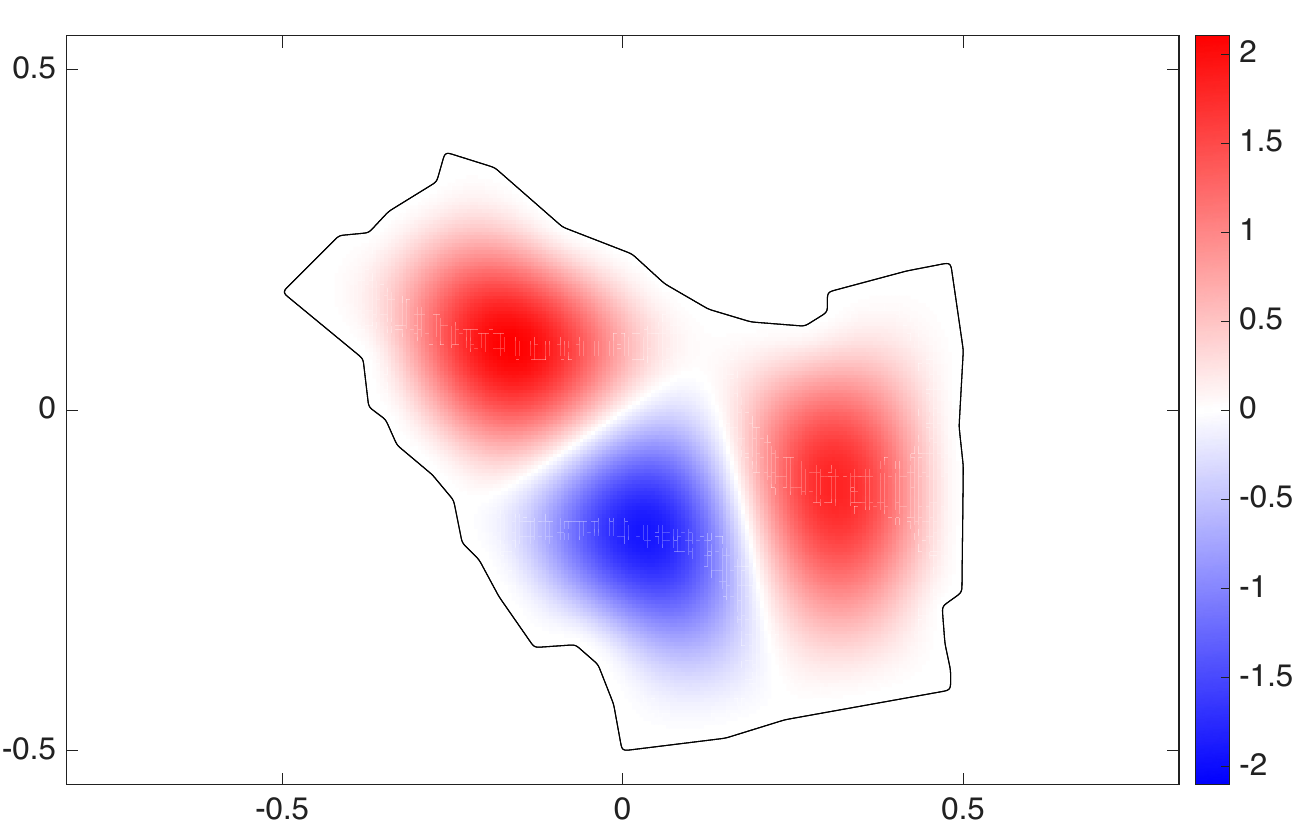}
        \caption{$\Omega_3$, $k=13.71$}
        \label{fig:eigenfunctions_h}
    \end{subfigure}
        \hfill
    \begin{subfigure}[t]{0.19\textwidth}
        \centering
        \includegraphics[width=\textwidth]{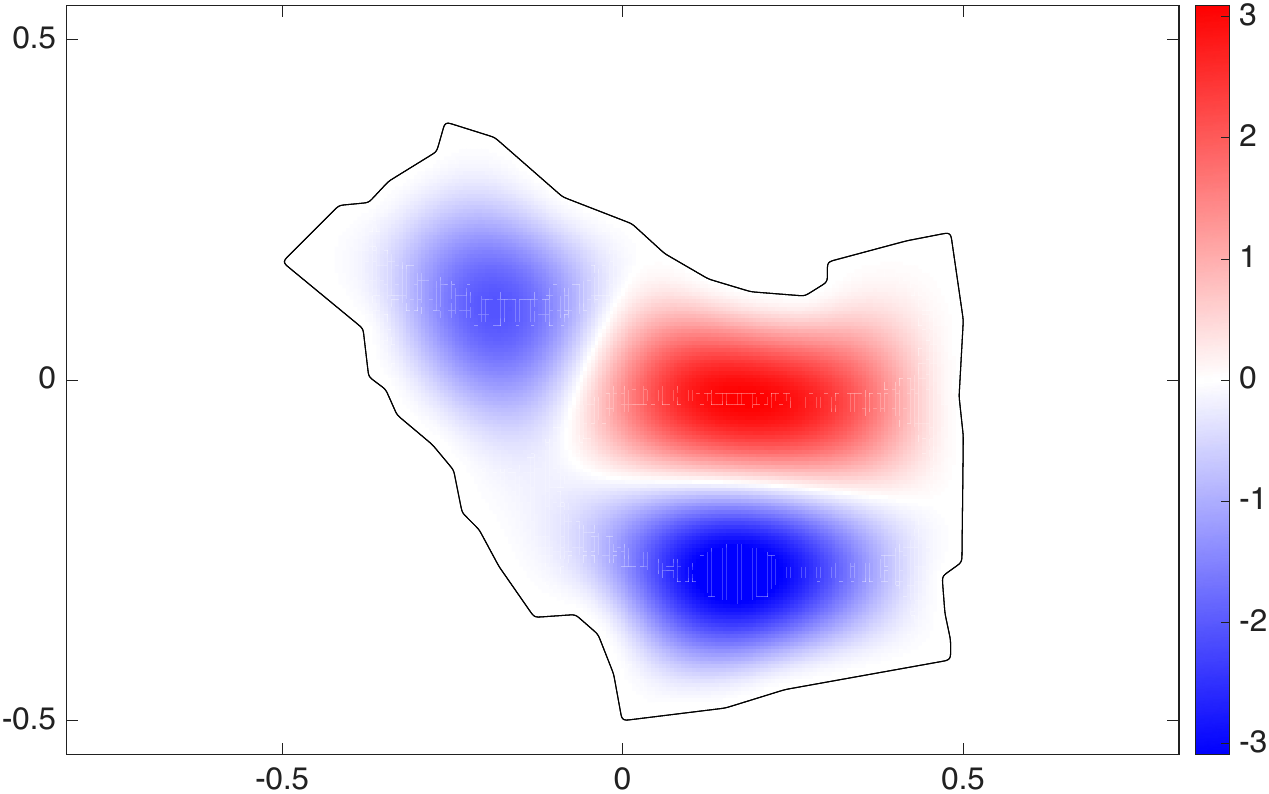}
        \caption{$\Omega_3$, $k=14.33$}
        \label{fig:eigenfunctions_i}
    \end{subfigure}
    \caption{Eigenfunctions of the clamped plate biharmonic operator on three different iceberg geometries. 
    Each subfigure shows an individual eigenfunction.
    As solutions of \cref{eq:biharmonic} with zero boundary data, both $u$ and $\partial_{\bn} u$ must vanish on the boundary, and $\Delta^2 u$ must hold on the interior.
    Statistics of the computation required to find these eigenfunctions are described in \cref{tab:eigenvalues}.
    }
    \label{fig:eigenfunctions}
\end{figure}

\subsection{Surface Helmholtz}
As our next example, we solve the Helmholtz-Beltrami problem on the ``chainlink torus'' surface $\Gamma$ depicted in \cref{fig:var_helmbel}:
\begin{equation}\label{eq:helmbel}
    (\Delta_\Gamma + k(\bx)^2)u(\bx) = f(\bx) \quad\text{on }\Gamma,
\end{equation}
where $\Delta_\Gamma$ is the Laplace-Beltrami operator. We solve this equation using a modification of the parametrix method introduced in \cite{goodwill_parametrix_2026}. In short, we suppose that 
\begin{equation*}
    u(\bx) = \int_\Gamma \frac{1}{2\pi}\log(\bx-\by) \sigma(\by) \, d\by,
\end{equation*}
for some unknown density $\sigma$. The calculations in \cite{goodwill_parametrix_2026} show that this $u$ will solve the Helmholtz-Beltrami equation provided
\begin{equation}
    \sigma(\bx)+\int_\Gamma R(\bx,\by)\sigma(\by)\,d\by = f(\bx)
\end{equation}
for all $\bx\in\Gamma$, where the remainder kernel is
\begin{equation*}
    R(\bx,\by) = \frac1{2\pi}\left[\lp\frac{\bn(\bx)\cdot(\bx-\by)}{\|\bx-\by\|^2}  \rp^2-2H(\bx) \frac{\bn(\bx)\cdot(\bx-\by)}{\|\bx-\by\|^2} + k^2(\bx) \log(\bx-\by)\right],
\end{equation*}
and $H(\bx)$ is the mean curvature of $\Gamma$. In \cite{goodwill_parametrix_2026}, the authors solve the same problem using a Hankel function with a characteristic frequency instead of the logarithm. We choose this representation as it is real-valued and requires a smaller $\nspread$. We discretize the integral equation using the collocation scheme \texttt{fmm3dbie} \cite{fmm3dbie}. In short, this package splits the surface into high order Koornwinder patches discretized with Vioreanu-Rokhlin nodes \cite{vioreanu2014spectra}. For every patch, the targets are split into a near-field where is special quadrature is used, and a far-field where a consistent oversampled smooth quadrature rule is used.  To apply the operators with our algorithm, we simply use the oversampled nodes as the sources and the original discretization as the target. We then add a sparse matrix of corrections to apply the near-field quadratures.

\begin{figure}[h!]
    \centering
        \includegraphics[width=0.7\textwidth]{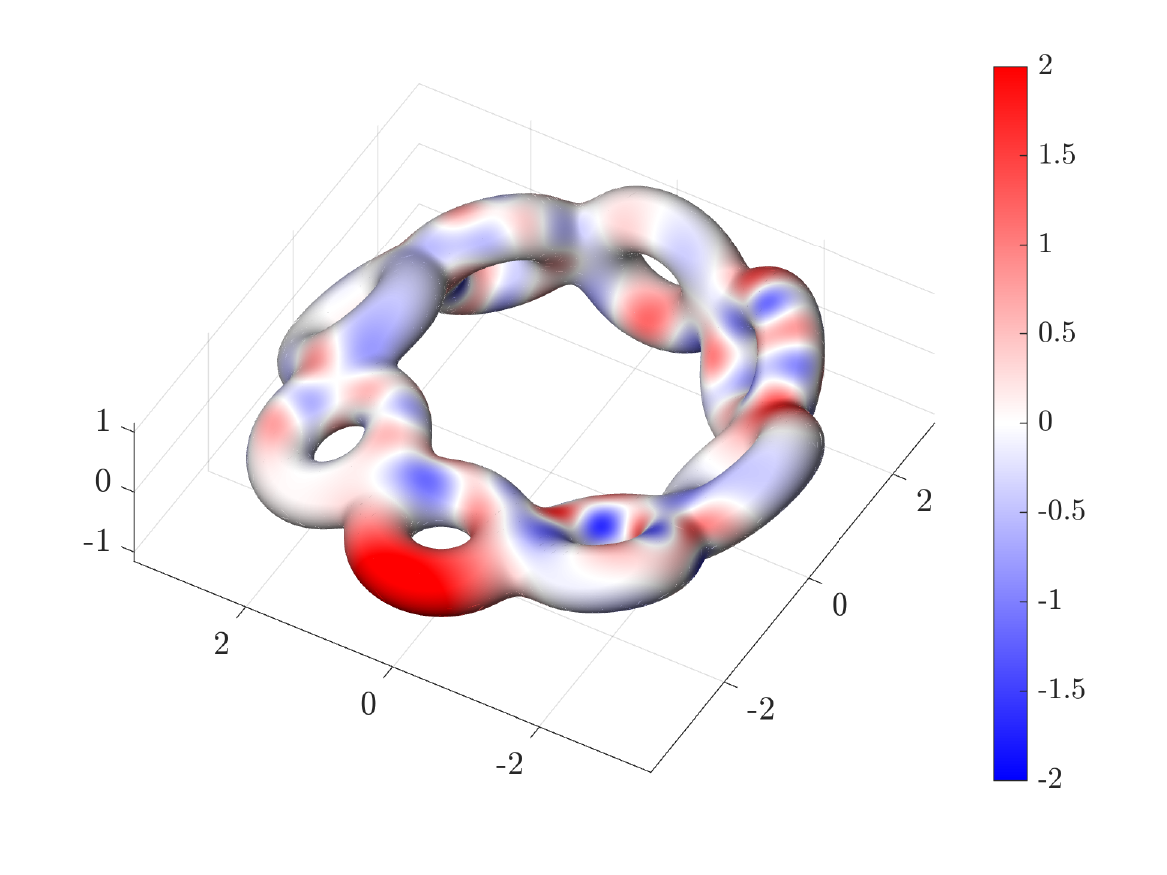}
    \caption{An example solution of the Helmholtz-Beltrami equation \cref{eq:helmbel}. The ``chainlink torus'' surface $\Gamma$ is a smooth surface triangulated with $7{,}797$ patches.}
    \label{fig:var_helmbel}
\end{figure}

For our experiment, we used $k(\bx)= 4-4\cos(2\bx_2)$; the source $f(\bx)$ and wavenumber $k(\bx)$ are shown in \cref{fig:helmbel_setup}. 
An analytic solution test indicates that the solution is accurate to 4 digits. 
The surface was discretized with $350{,}865$ points belonging to the $7{,}797$ $8^{\text{th}}$ order patches. It took $26$ seconds to generate the quadratures. For the far-field quadrature \texttt{fmm3dbie} oversamples the surface to $14^{\text{th}}$ order, resulting $674{,}924$ nodes. The precorrected-FFT precomputation with these sources and targets took $181$ seconds for the remainder and $180$ seconds for the parametrix. Each application of the system matrix took $1.50$ seconds. The solve required $365$ GMRES iterations which took $899$ seconds. Overall the entire solver used $48$ GB of memory.

\begin{figure}[h!]
    \centering
        \includegraphics[width=\textwidth]{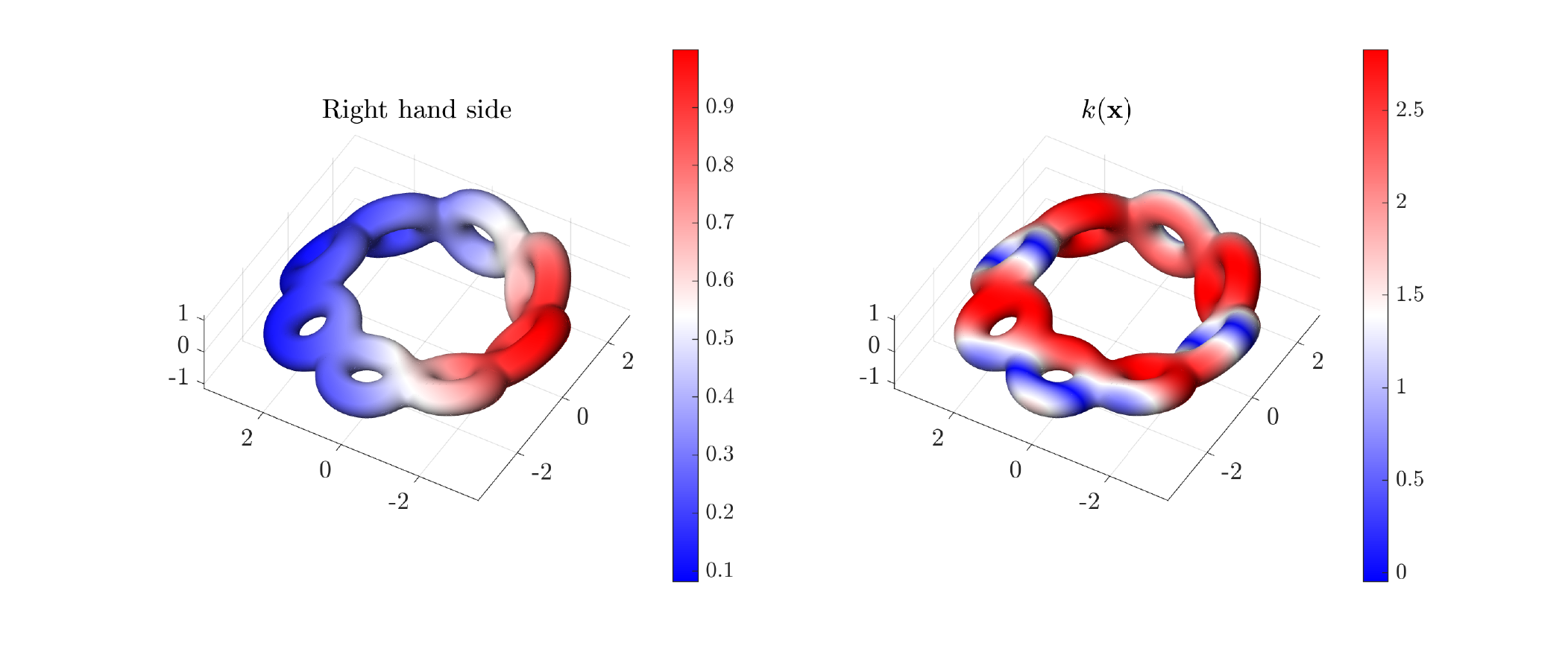}
    \caption{The example right hand side source function $f(\bx)$ (left) and wavenumber $k(\bx)$ (right) used to generate the Helmholtz-Beltrami solution in \cref{fig:var_helmbel}.}
    \label{fig:helmbel_setup}
\end{figure}

\subsection{Lens optimization}
As a final example, we consider a simple lens design problem inspired by \cite{vico2025designing}. In that work,  3D printing is used to build a lens out of a dielectric material with a variable density. In this example we consider a simpler 2D lens and model the field using a scalar Helmholtz equation.  We model this as a compact region $\Omega$ with a local wavenumber $k(\bx)=\sqrt{k_1^2 + q(\bx)}$ that is customizable between some maximum and minimum allowable values. Outside of the lens, we suppose that the field has wavenumber $k_0$.

Our goal will be to choose the wavenumber so that an incident planewave $u^{\rm{in}}$ is focused to yield the largest possible intensity at the focal point $\bx_0$. The total field can be split $u=u^{\rm{in}} + u^{\rm{sc}}$, where the scattered field satisfies the Sommerfeld radiation conditions and the equations
\begin{equation}\label{eq:lens_eq}
\left\{
    \begin{aligned}
        \Delta u^{\rm{sc}}(\bx) + k_0^2 u^{\rm{sc}}(\bx) &= 0, && \quad \bx \text{ in }\bbR^2\setminus\Omega,\\
        \Delta u^{\rm{sc}}(\bx) + (q(\bx)+k_1^2) u^{\rm{sc}}(\bx) &= -(q(\bx)+k_1^2 - k_0^2)u^{\rm{in}}(\bx),&& \quad \bx \text{ in }\Omega,\\
        [[u^{\rm{sc}}]](\bx)&=0, && \quad \bx \text{ on }\partial\Omega,\\
        [[\partial_{\bn} u^{\rm{sc}}]](\bx)&=0, && \quad \bx \text{ on }\partial\Omega,
    \end{aligned}
    \right.
\end{equation}
where $[[u^{\rm{sc}}]]$ and $[[\partial_{\bn} u^{\rm{sc}}]]$ are the jump in the scattered field and its normal derivative across $\partial \Omega$ respectively. This is equivalent to looking for the weak (i.e. $H^2_{\rm{loc}}$) solution of the equation 
\begin{align*}
(\Delta + k^2(\bx))u^{\rm{sc}}(\bx)=-(k(\bx)^2 - k_0^2)u^{\rm{in}}(\bx)
\end{align*}
in all of $\bbR^2$ that satisfies the Sommerfeld radiation conditions.
 We will work to find a $q(\bx)$ that optimizes the intensity at the focal point $\bx_0$:
\begin{equation*}
    \cF[q] := |u^{\rm{inc}}(\bx_0)+u^{\rm{sc}}(\bx_0)|^2
\end{equation*}
where $\bx_0$ is some point outside of $\Omega$. For this example, we choose $\Omega$ to be the interior of the squircle satisfying
\begin{align*}
    x_1^2+x_2^2 - 0.9^2 x_1^2x_2^2 = 1
\end{align*}
which is shown in \cref{fig:lens}. The focal point is chosen as $\bx_0 = (4, 0)$.
We represent $q$ using a global $6^{\text{th}}$ order tensor product Chebyshev basis, supported on the square that just contains the region $\Omega$. To improve the conditioning, we use a QR factorization (weighted by the square root of the quadrature weights in $\Omega$) to orthogonalize the basis on $\Omega$. 

To solve this optimization problem efficiently, we must efficiently compute the Fr\'echet derivative of the scattered field with respect to perturbations of $q$. If $D[u^{\rm{sc}}]p$ is the Fr\'echet derivative of $u^{\rm{sc}}$ with respect to a perturbation $q(\bx)\to q(\bx) + h p(\bx)$, it will be the outgoing weak solution of 
\begin{align*}
    (\Delta + k^2(\bx))(D[u^{\rm{sc}}]p)(\bx)=-p(\bx)(u^{\rm{in}}(\bx)+u^{\rm{sc}}(\bx))
\end{align*}
(see \cite{borges2017high}). The strong form of this equation is identical to \cref{eq:lens_eq} with a new volumetric source in $\Omega$. To simplify the computation, we introduce $w$, the weak solution of
$$ (\Delta+k^2(\bx)) w(\bx) = \delta(\bx-\bx_0)$$
(i.e. the field generated by a point source at the focal point $\bx_0$). A straightforward integration by parts argument \cite{borges2017high} then shows that
\begin{multline*}
    (D[u^{\rm{sc}}]p)(\bx_0) = v(\bx_0) = \int_{\bbR^2} (D[u^{\rm{sc}}]p)(\bx)\delta(\bx-\bx_0)\,d\bx = \int_{\bbR^2} (D[u^{\rm{sc}}]p)(\bx)(\Delta+k^2(\bx)) w(\bx)\,d\bx \\=\int_{\bbR^2} (\Delta+k^2(\bx))  (D[u^{\rm{sc}}]p)(\bx)w(\bx)\,d\bx =-\int_{\Omega} p(\bx)(u^{\rm{in}}(\bx)+u^{\rm{sc}}(\bx))w(\bx)\,d\bx.
\end{multline*}
The advantage of this approach is that we can compute the partial derivative of the field at $\bx_0$ with respect to all of the Chebyshev coefficients defining $q$ with two solves, one to find $u^{\rm{sc}}$ and one to find $w$, followed by one inner product per basis function. 
After some applications of the chain rule, this approach allows us to compute the gradient of the objective function with respect to the coefficients defining $q$.
This in turn allows us to optimize the objective function $\cF[q]$ using BFGS \cite{broyden1970convergence,fletcher1970new,shanno1970conditioning,goldfarb1970family} and a backtracking line search with the initial condition $q(\bx)\equiv 0$.

In order to solve \cref{eq:lens_eq} at each step, we represent the scattered field using a volume potential $\sigma$ supported in $\Omega$ and layer potentials $\eta$ and $\rho$ supported on $\partial\Omega$:
\begin{equation*}
    u^{\rm{sc}}(\bx) = \int_{\partial\Omega}G_{k_0}(\bx-\by) \eta(\by) \,d\by\\+ \int_{\partial\Omega}\partial_{\bn(\by)}G_{k_0}(\bx-\by) \rho(\by) \,d\by\quad \bx\in\bbR^2\setminus \Omega,
\end{equation*}
where $G_{k_0}(\bx)=\frac{i}4 H^{(1)}_0(k_0\|\bx\|
)$ is the Green's function for the exterior problem,
and
\begin{equation*}
    u^{\rm{sc}}(\bx) = \int_\Omega G_{k_1}(\bx-\by) \sigma(\by) \,d\by+ \int_{\partial\Omega}G_{k_1}(\bx-\by) \eta(\by) \,d\by\\+ \int_{\partial\Omega}\partial_{\bn(\by)}G_{k_1}(\bx-\by) \rho(\by) \,d\by\quad \bx\in \Omega,
\end{equation*}
where $G_{k_1}$ is the Green's function for the prototypical interior wavenumber. 
Plugging this representation into \cref{eq:lens_eq} yields a Fredholm second kind equation that can solved for $\sigma,\eta,\rho$. We then discretize $\Omega$ using \texttt{fmm3dbie} and $\partial \Omega$ using \texttt{chunkie}. We apply the operator quickly using two precorrected-FFTs, one at wavenumber $k_1$, and a small one involving only boundary nodes at wavenumber $k_0$. Since $k_0$ and $k_1$ are fixed, the quadrature corrections, spreading operators, and near-field corrections can be pre-computed and reused at each optimization step.

It is worth remarking that \cref{eq:lens_eq} can also be solved by using the much simpler representation $u^{\rm{sc}}(\bx) = \int_\Omega G_{k_0}(\bx-\by) \sigma(\by) \,d\by$ for all $\bx\in\bbR^2$, but we have found that our representation requires substantially fewer GMRES iterations for small $q$ (33 versus 66 iterations for $(k_0,k_1)=(4,6)$ and $q(\bx)\equiv-1$ in $\Omega$), likely because our representation uses a more representative wavenumber $k_1$ inside $\Omega$. We have found that the decreased number of iterations outweighs the increased cost associated with adding discretization points on the boundary and the second small precorrected-FFT per GMRES iteration.

With our integral equation solver in hand, we run our BFGS optimizer to find an improved lens with $k_0=4$, $k_1=6$, and $q$ restricted to live in $[-10,10]$. We stopped the optimization after 132 iterations, when the norm of the gradient first dropped below $10^{-3}$.
The resulting field intensity and wavenumber are shown in \cref{fig:lens}. The intensity at $\bx_0=(4,0)$ is $|u^{\rm{in}}(\bx_0)+u^{\rm{sc}}(\bx_0)|^2=5.1$. For comparison, the best constant wavenumber $k\in[5,7]=k_1+[-1,1]$ was $5$, which generated an intensity of $4.79$. The discretization has $20{,}138$ unknowns and used an average of 60.2 GMRES iterations per solve with a tolerance of $10^{-9}$. The total time to solve the optimization problem was $2{,}041$ seconds.

\begin{figure*}[t!]
    \centering
    \begin{subfigure}[t]{0.45\textwidth}
        \centering
        \includegraphics[width=\textwidth]{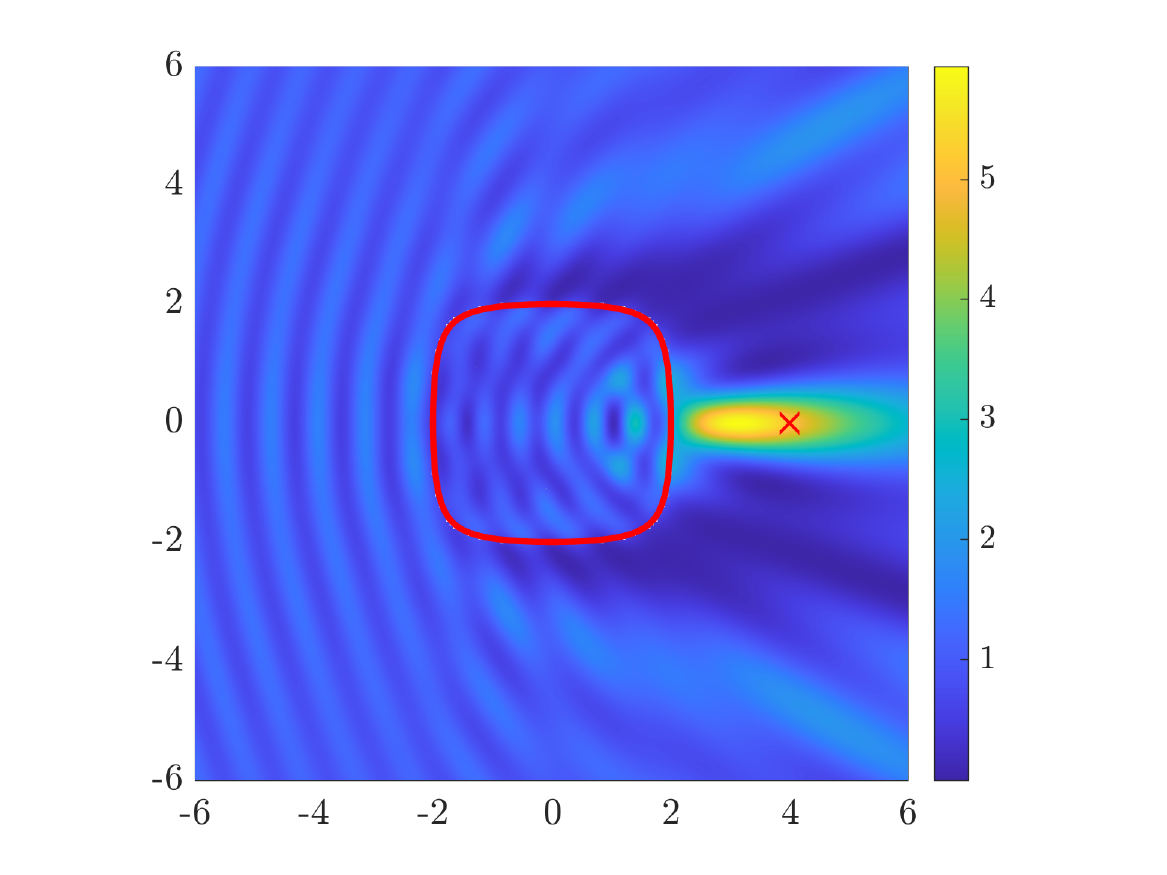}
        \caption{The field intensity $|u^{\rm{in}}+u^{\rm{sc}}|^2$.}
        \label{subfig:lens_u}
    \end{subfigure}
    ~
    \begin{subfigure}[t]{0.45\textwidth}
        \centering
        \includegraphics[width=\textwidth]{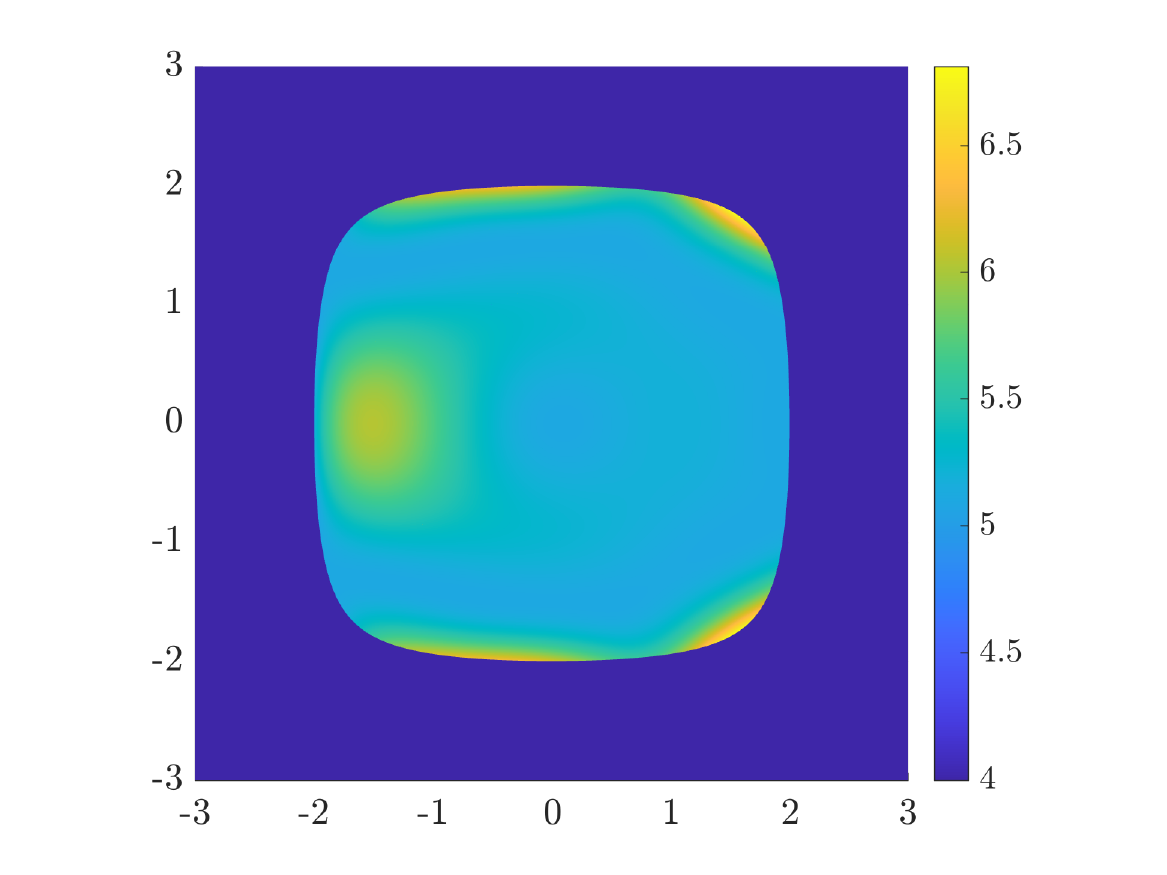}
        \caption{The wavenumber $k(\bx)$.}
        \label{subfig:lens_q}
    \end{subfigure}
    \caption{The field intensity $|u^{\rm{in}}+u^{\rm{sc}}|^2$ (\ref{subfig:lens_u}) and wavenumber $k(\bx)$ (\ref{subfig:lens_q}) for the lens chosen to maximize the intensity at $\bx_0=(4,0)$ when $u^{\rm{in}}$ is a plane wave coming from the left.}
    \label{fig:lens}
\end{figure*}

\section{Conclusion}\label{sec:conc}
In this work, we present a modification of the precorrected-FFT method that allows it to be used for a generic kernels. The method uses a series of proxy shells to build spreading operators that map point charges to equivalent charges on a grid. We also present and test an algorithm for systematically choosing the proxy points and spreading grids without the need for analytical knowledge of the kernel. The method is therefore flexible and can trivially be applied to new problem. To demonstrate the effectiveness of the new flexible version of the precorrected-FFT, we use it to solve accelerate integral equation approach for a variety of PDE problems. 

So far, this work has been limited to  scalar-valued kernels. While the presented algorithm should extend directly to vectored-valued problems, such as those considered in \cite{tissari_precorrected-fft_2003}, efficient implementations remains an open question. In order to offset the extra cost for these problems, we plan to explore GPU implementations of the algorithm.

\section*{Acknowledgments}
The authors would like to thank Jeremy Hoskins, Travis Askham, and Manas Rachh for many helpful discussions on the algorithm and for suggesting that we complete this work. They would also like to thank Shidong Jiang for his suggestions of good benchmark problems and Alex Barnett and Vicente Gomez Herrera for help conversations. T. Goodwill was was supported in part by an AMS-Simons Travel Grant.
The Flatiron Institute is a division of the Simons Foundation.

\bibliography{ref.bib}

\end{document}